\documentclass[11pt]{amsart}
\usepackage{amssymb,amsfonts,graphicx,amsmath,amsthm,amstext,latexsym,amsfonts}
\usepackage{graphicx,epsfig}
\usepackage{bbm}
\usepackage{t1enc}
\usepackage{mathrsfs}
\usepackage{subfig}
\usepackage{hyperref}
\usepackage{a4wide}
\usepackage{tikz}
\usetikzlibrary{intersections}
\usepackage{egothic}
\usepackage [T1] {fontenc}
\theoremstyle{plain}

\numberwithin{equation}{section}
\newtheorem{theorem}{Theorem}[section]
\newtheorem{proposition}[theorem]{Proposition}
\newtheorem{lemma}[theorem]{Lemma}
\newtheorem{definition}[theorem]{Definition}
\newtheorem{remark}[theorem]{Remark}

\newenvironment{proofad2}{\removelastskip\par\medskip
\noindent{\textbf {Proof of Theorem \ref{mainth1comp}}.}
\rm}{\penalty-20\null\hfill$\blacksquare$\par\medbreak} 

\newenvironment{proofad3}{\removelastskip\par\medskip
\noindent{\textbf {Proof of Theorem \ref{mainth3}}.}
\rm}{\penalty-20\null\hfill$\blacksquare$\par\medbreak} 

\newenvironment{proofad4}{\removelastskip\par\medskip
\noindent{\textbf {Proof of Theorem \ref{rotheorem}}.}
\rm}{\penalty-20\null\hfill$\blacksquare$\par\medbreak} 

\usepackage{color}
\definecolor{darkred}{rgb}{0.8,0,0}
\definecolor{darkblue}{rgb}{0,0,0.7}
\definecolor{darkgreen}{rgb}{0,0.4,0}

\newcommand{\eps}{\varepsilon}

\newcommand{\R}{{\mathbb R}}

\newcommand{\W}{{\mathcal W}}

\newcommand{\V}{{\mathcal V}}

\newcommand{\un}{{\rm 1\kern -2.5pt l}}
\newcommand{\tr}{{\rm Tr}}

\def\w{\mathbf{w}}
\def\u{\mathbf{u}}

\def\vv{\mathbf{v}}
\def\yy{\mathbf{y}}
\def\n{\mathbf{n}}

\def\vv{\mathbf{v}}

\def\eps{\varepsilon}
\def\R{{\mathbb R}}

\def\H{{\mathcal H}}

\def\eps{\varepsilon}
\def\R{{\mathbb R}}

\def\H{{\mathcal H}}

\def\argmax{\mathop{{\rm argmax}}\nolimits}
\def\argmin{\mathop{{\rm argmin}}\nolimits}

\def\Tr{\mathop{{\rm Tr}}\nolimits}

\def\dv{\mathop{{\rm div}}\nolimits}

\def\u{\mathbf{u}}
\def\v{\mathbf{v}}

\def\v{{\bf v}}
\def\w{{\bf w}}

\def\x{{\bf x}}

\def\Id{\mathbf{I}}

\def\wconv{\rightharpoonup}

\newcommand{\MMM}{\color{black}}
\newcommand{\KKK}{\color{black}}
\newcommand{\PPP}{\color{magenta}}
\newcommand{\LLL}{\color{blue}}

\renewcommand{\epsilon}{\varepsilon}
\newcommand{\beeq}{\begin{equation}}
\newcommand{\eneq}{\end{equation}}
\newcommand{\bear}{\begin{array}}
\newcommand{\enar}{\end{array}}
\newcommand{\bema}{\begin{displaymath}}
\newcommand{\enma}{\end{displaymath}}

\newcommand{\beea}{\begin{eqnarray}}
\newcommand{\enea}{\end{eqnarray}}

\newcommand{\om}{\Omega}

\newcommand{\bb}{\boldsymbol}

\newcommand{\lab}[1]{ \label{#1} }

\newenvironment{proofth1}{\removelastskip\par\medskip   
\noindent{\bf Proof of {\rm {\bf Theorem \ref{mainth1}}}.}
\rm}{\penalty-20\null\hfill$\blacksquare$\par\medbreak} 

\def\Id{\mathbf{I}}

\def\wconv{\rightharpoonup}

\title[Sharp conditions for the linearization of finite elasticity]{Sharp conditions for the linearization of finite  elasticity}
  \author{Edoardo Mainini}
\address[Edoardo Mainini]{Dipartimento di Ingegneria meccanica, energetica, gestionale e dei trasporti, 
  Universit\`a  degli studi di Genova, Via all'Opera Pia, 15 - 16145 Genova Italy.}
\email{mainini@dime.unige.it}

\author{Danilo Percivale}
\address[Danilo Percivale]{Dipartimento di Ingegneria meccanica, energetica, gestionale e dei trasporti, 
  Universit\`a  degli studi di Genova, Via all'Opera Pia, 15 - 16145 Genova Italy.}
\email{percivale@dime.unige.it}

\subjclass[2010]{49J45, 74K30, 74K35, 74R10}
\keywords{Calculus of Variations, 
  Linear Elasticity, Finite Elasticity,   
  Gamma-convergence,  incompressibility}
\begin{document}
 \maketitle
\begin{abstract}
We consider the topic of linearization of finite elasticity for pure traction problems. We characterize the variational limit for the approximating sequence of rescaled nonlinear elastic energies. We show that the limiting minimal value  can be strictly lower than the minimal value of the standard linear elastic energy if a strict compatibility condition for external loads does not hold.
The results are provided for both the compressible and the incompressible case.
\end{abstract}


\section{Introduction}

 Let $\Omega\subset\mathbb R^3$ be the reference configuration of a hyperelastic body. If  $\mathbf y:\Omega\to\mathbb R^3$ is the deformation field  and  $h > 0$ is an adimensional  parameter,   we introduce the scaled global energy of the body, including the stored elastic energy and the work of external forces,  by
\begin{equation}\label{un}
\displaystyle \mathcal G_h(\mathbf y):=h^{-2}\int_\om\mathcal W(\x,\nabla\mathbf y)\,d\x
-h^{-1}\mathcal L(\mathbf y-\mathbf i).
\end{equation}
Here, $\mathcal W:\Omega\times\mathbb R^{3\times 3}\to[0,+\infty]$ is the strain energy density and
 $\mathbf i$ denotes the identity map.
 For every $\x\in\om$, the function $\mathcal W(\x,\cdot)$ is assumed to be frame indifferent and uniquely minimized at rotations with value $0$.  
 It is also assumed to \MMM be $C^2$-smooth around rotations and to satisfy \KKK a suitable coercivity condition to be introduced later on.
  Moreover, the load functional $\mathcal L$ is defined by
 $$\mathcal L(\v):=\int_\om\mathbf f\cdot\v\,d\x+\int_{\partial\Omega}\mathbf g\cdot \v\,d\mathcal H^2(\x),$$
 where $\mathbf f: \Omega\to\mathbb R^3$ is a volume force field, $\mathbf g:\partial\om\to\mathbb R^3$ is a surface force field, and $\mathcal H^2$ denotes the surface measure. 
 
 \medskip

 In a pure traction problem,
  in order to study stable equilibrium configurations we have to assume that  \KKK
 \begin{equation*}\lab{precomp}
\mathcal L(\mathbf y-\mathbf i)\le 0
\end{equation*}
 for every deformation $\mathbf y$ such that 
 \begin{equation}\lab{kernell}\int_\om \mathcal W(\x,\nabla\mathbf y)\,d\x=0.\end{equation}
 \MMM Under our assumptions on $\mathcal W$, \eqref{kernell}   holds true if and only if $\nabla \mathbf y$ is a constant rotation matrix, i.e., \KKK $\mathbf y(\x)=\mathbf R \x+\mathbf c$ for some $\mathbf R\in SO(3)$ and some $\mathbf c\in \mathbb R^3$, where $SO(3)$ denotes the special orthogonal group.
 Thus, we need to assume that 
 \begin{equation}\label{rc}
 \mathcal L( (\mathbf R - \mathbf I) \x+\mathbf c)\le 0
\end{equation}
for every $\mathbf R\in SO(3)$ and every $\mathbf c\in \mathbb R^3$,  and by taking $\mathbf R= \mathbf I$ we get  $\mathcal L(\mathbf c)\le 0$ for every $\mathbf c\in \mathbb R^3$, that is, $\mathcal L(\mathbf c)= 0$ for every $\mathbf c\in \mathbb R^3$. \MMM Hence, \eqref{rc} is equivalent to the following two conditions
\begin{equation}\label{c}
 \mathcal L(\mathbf c)= 0\ \qquad \forall\  \mathbf c\in \mathbb R^3,
\end{equation}
\begin{equation}\label{wcomp} \mathcal L( (\mathbf R - \mathbf I) \x)\le 0\ \qquad \forall\ \mathbf R\in SO(3).
\end{equation}
We observe that 
if  $\mathbf R\in SO(3)$ exists such that  
$\mathcal L((\mathbf R-\mathbf I)\x)> 0,$
then  $\mathcal G_h$ is not uniformly bounded from below with respect to $h$, that is, $\inf \mathcal G_h\to -\infty$  as $h\to 0$ (see also Remark 2.9 below).
It is worth noting that \eqref{c} says that  external loads have null resultant while it will be shown in Remark \ref{rem21}  that \eqref{wcomp} implies they have null momentum (without being equivalent to the null momentum condition). 

\medskip

The choice of the scaling powers in \eqref{un} depends on the behavior of the elastic strain energy density and of the work expended by external loads for deformations which are close to a suitable rotation of the reference configuration, say $\mathbf y=\mathbf R(\mathbf i+h\u)$, where $\u:\om\to\mathbb R^3$ and where $\mathbf R$ belongs to the following {\it rotation kernel} associated to $\mathcal L$ (that satisfies \eqref{c}-\eqref{wcomp})
\begin{equation}\label{S0}
\mathcal S^0_{\mathcal L}:=\{\mathbf R\in SO(3): \mathcal L((\mathbf R-\mathbf I)\x)=0\}.
\end{equation}
Indeed, by frame indifference   we obtain
\begin{equation}\label{G_hh}
\mathcal G_h(\mathbf y)=h^{-2}\int_\om\mathcal W(\x,\mathbf I+h\nabla\u(\x))\,d\x-\mathcal L(\mathbf R\u)-h^{-1}\mathcal L((\mathbf R-\mathbf I)\x)
,
\end{equation}
and if $\mathbf R\in \mathcal S^0_{\mathcal L}$, by a Taylor expansion of $\mathcal W(\x,\cdot)$ around the identity matrix we  formally \KKK get for every fixed $\u$
\begin{equation}\lab{heuristic1}
\displaystyle \lim_{h\to 0} \mathcal G_h(\mathbf y)= \displaystyle\int_\Omega\mathcal Q(\x,\mathbb E(\u))\,d\x -\mathcal L(\mathbf R\u),
\end{equation}
where $\mathbb E(\u):=\tfrac12(\nabla\u^T+\nabla\u)$ and  where we have introduced the quadratic form \KKK
 \[\mathcal Q(\x,\mathbf F):=
 \frac12\,\mathbf F^T\,D^2 {\mathcal W}(\x,\mathbf I)\,\mathbf F,\qquad\mathbf F\in\mathbb R^{3\times3},\quad\x\in\om.\]
Therefore, it is  natural to guess that the variational limit $\mathcal G$ of $\mathcal G_h$ as $h\to0$ can be obtained from \eqref{heuristic1} through a minimization among all $\mathbf R\in \mathcal S^0_{\mathcal L}$, namely
\begin{equation}\label{heuristicbis}
\displaystyle {\mathcal G}(\u):=\int_\om\mathcal Q(\x,\mathbb E(\u))\,d\x
-\max_{\mathbf R\in \mathcal S^0_{\mathcal L}} \mathcal L(\mathbf R\u).
\end{equation}
We stress that, if $\mathbf R=\mathbf I$,  this corresponds to the usual formal  derivation of linearized elasticity. But for any other $\mathbf R\in \mathcal S^0_{\mathcal L}$ the work done by external loads for going from $\Omega$ to $\mathbf R\Omega$ is null, so that it might be energetically convenient to consider  deformations near  $\mathbf R\x$ rather than near the identity. On the other hand it is clear that if $\mathbf R\notin \mathcal S^0_{\mathcal L}$ this heuristic argument fails.  Indeed, choosing $\mathbf R \in SO(3)\setminus \mathcal S^0_{\mathcal L}$ \KKK
  is not energetically convenient  due to the behavior of the last term in the right hand side of \eqref{G_hh} as $h\to 0$\KKK.

\medskip

In the case that \eqref{c}-\eqref{wcomp} hold and  $\mathcal S^0_{\mathcal L}\equiv\{ \mathbf I\}$, 
then \eqref{heuristicbis} reduces to
\begin{equation*}
\mathcal E(\u):= \displaystyle\int_\Omega\mathcal Q(\x,\mathbb E(\u))\,d\x -\mathcal L(\u)
\end{equation*}
which is  the standard form of the \KKK total potential energy of the elastic body   in the linear setting. \KKK
\KKK
It has been shown in \cite{MPTARMA} (see also \cite{MPTJOTA}) that  in this case
\begin{equation}\label{heuristic} \lim_{h\to 0}(\inf\mathcal G_h)= \min\mathcal E \end{equation}
 and if $\mathcal G_h(\mathbf y_h)-\inf \mathcal G_h\to 0$ as $h\to 0$ (i.e., if $(\mathbf y_h)$ is a sequence of quasi-minimizers of $\mathcal G_h$) then \begin{equation}\label{heuristic2}\u_h:=h^{-1}(\mathbf y_h-\mathbf i)\to \u_0\in \argmin\mathcal E\end{equation} in a suitable sense.  In particular,  \KKK $\u_0$ satisfies the equilibrium conditions
\begin{equation}\lab{equilibrium}\left\{\begin{array}{ll}
-\dv \mathcal Q'(\x,\mathbb E(\u_0))=\mathbf f\qquad &\hbox{ in}\ \om\\
&\\
\mathcal Q'(\x,\mathbb E(\u_0))\, \n=\mathbf g\qquad &\hbox{ on}\ \partial\om,
\end{array}\right.
\end{equation} 
where $\mathcal Q'(\x,\mathbf F):=D^2\mathcal W(\x,\mathbf I)\,\mathbf F$ and $\mathbf n$ is the outer unit normal to $\partial\om$. \KKK
In \cite{MP2} we have extended the results of  \cite{MPTARMA} to  incompressible elasticity.  Indeed, it is shown in \cite{MP2} that \eqref{heuristic} and \eqref{heuristic2} hold true \KKK 
 by substituting $\mathcal G_h$ with
 the scaled incompressible global energy  $\mathcal G_h^I$, defined by replacing $\mathcal W$ with $\mathcal W^I$ in the right hand side of \eqref{un}, 
where
 \begin{equation*}
{\mathcal W}^I (\x, \mathbf F):=\left\{\begin{array}{ll} \mathcal W(\x,\mathbf F) \qquad &\hbox{if}\ \ \det\mathbf F=1\\
 +\infty\ &\hbox{otherwise},
\end{array}\right.
\end{equation*}
and by substituting  $\mathcal E$ with
 \begin{equation*}
\mathcal E^I(\v):=\int_\om\mathcal Q
^I(\x,\mathbb E(\v))-\mathcal L(\v),\end{equation*}
where
\begin{equation}\label{QI}
{\mathcal Q}^I (\x, \mathbf F):=\left\{\begin{array}{ll} \vspace{0.2cm}\dfrac12\,\mathbf F^T\,D^2\mathcal W(\x,\mathbf I)\,\mathbf F\ \qquad &\hbox{if}\ \ \tr\,\mathbf F=0\\
 +\infty\ &\hbox{otherwise.}
\end{array}\right.
\end{equation}
 Roughly speaking, these results can be interpreted by saying that, if \eqref{c}  holds along with the strict compatibility condition
\[ \mathcal L( (\mathbf R - \mathbf I) \x)< 0\ \qquad \forall\ \mathbf R\in SO(3)\setminus\{\mathbf I\},
\]
  then linear elasticity can be viewed as the variational limit of finite elasticity both in the compressible and in the incompressible case. 

\medskip

By assuming only \eqref{c} and  \eqref{wcomp}, since $\mathcal S_{\mathcal L}^0$ needs not be reduced to the identity matrix, a minimizer  of functional \eqref{heuristicbis} is not expected to satisfy \eqref{equilibrium} in general. On the other hand,  we may ask if its energy level equals the minimal value of $\mathcal E$: this fact
 is still an  open question and it represents our main focus, along with the analogous comparison between optimal energy levels in the incompressible case.
\KKK


\medskip

  A consequence of a recent result  shown by Maor and Mora in    \cite[Theorem 5.3]{MM} is that if \eqref{c} and \eqref{wcomp} 
hold along with a quadratic growth condition from below for $\mathcal W(\x,\cdot)$, then  indeed \KKK
$$\lim_{h\to 0} (\inf \mathcal G_h) = \min \mathcal G.$$
In this paper we extend this result to the incompressible case and to more general coercivity assumptions  on the strain energy density $\mathcal W$,
but more than anything else we  exhibit examples in which 
\begin{equation}\label{new}\min \mathcal G < \min \mathcal E\qquad\mbox{and}\qquad \min \mathcal G^I < \min \mathcal E^I,\end{equation}
where $\mathcal G^I$ is defined by replacing $\mathcal Q$ with $\mathcal Q^I$ in \eqref{heuristicbis}.
{\MMM  
Surprisingly enough, this shows that,  at least under the sole assumptions \eqref{c}-\eqref{wcomp}, the energy level of the minimizer of $\mathcal E$  does not necessarily provide the  minimal value \KKK of the variational limit of the scaled  finite elasticity functional $\mathcal G_h$.
 A
  gap between \KKK $\lim_{h\to 0}(\inf\mathcal G_h)$ (resp. $\lim_{h\to 0}(\inf\mathcal G_h^I)$) and $\min\mathcal E$ (resp. $\min\mathcal E^I$) may appear. 

\medskip

In detail, by \KKK assuming the coercivity condition
\begin{equation}\label{pgr}
\begin{array}{ll}
\W(\x,\mathbf F)\ge  C \ g_{p}(d(\mathbf F, SO(3)))\qquad
 \forall\, \mathbf F\in \mathbb R^{3 \times 3}
\end{array}
\end{equation}
for some $p\in(1,2]$, where
\begin{equation}\label{gp}
g_{p}(t):=\left\{\begin{array}{ll} \!\! t^{2}\quad &\hbox{if}\ 0\le t\le 1\vspace{0.3cm}\\
\!\! \displaystyle \frac{2t^{p}}{p}-\frac{2}{p}+1\quad &\hbox{if}\ t\ge 1\\
\end{array}\right.
\end{equation}
and $d(\cdot,SO(3))$ denotes the distance function from rotations, for the incompressible case we will prove the following result  (see Theorem \ref{mainth1} below).
If $(\mathbf y_h)\subset W^{1,p}(\om,\mathbb R^3)$ is a sequence of quasi-minimizers of $\mathcal G_{h}^I$, 
then by defining the {\it generalized rescaled displacements}
 \begin{equation*}
 \u_h(\x):= h^{-1}(\mathbf R_h^T\mathbf y(\x)-\x), \;\;\quad\mbox{where}\quad\;\;
\mathbf R_h\in \argmin\left\{ \int_\om g_p(|\nabla\mathbf y_h-\mathbf R|)\,d\x: \mathbf R\in SO(3)\right\},
\end{equation*}\KKK
there is a (not relabeled) subsequence such that
\[ \nabla\u_{h}\wconv \nabla\u_* \ \hbox{weakly in}\ L^{p}(\Omega,\R^{3\times3})\qquad\mbox{as $h\to 0$},
\]
where $\u_*\in H^1(\om,\R^3)$ and $\u_*$ is a minimizer of $\mathcal G^I$ over $W^{1,p}(\om,\mathbb R^3)$. Moreover,
\begin{equation*}   \mathcal G_h^I(\mathbf y_{h})\to \mathcal G^I(\u_*)\qquad\mbox{and}\qquad
\inf_{W^{1,p}(\om,\mathbb R^3)}\mathcal G_{h}^I\to \min_{W^{1,p}(\om,\mathbb R^3)}\mathcal G^I\quad\qquad\mbox{as $h\to 0$}.
\end{equation*}
Here,  the precise characterization of $\mathcal G^I$ is
\begin{equation*}
\displaystyle {\mathcal G}^I(\u)=\left\{\begin{array}{ll}\displaystyle \int_\om \mathcal Q^I(\x, \mathbb E(\u))\,d\x-\max_{\mathbf R\in \mathcal S^0_{\mathcal L}} \mathcal L(\mathbf R\u)\quad &\hbox{if} \ \u\in H^1_{\dv}(\om,\mathbb R^3)\\
&\\
 \ \!\!+\infty\quad &\hbox{otherwise in} \ W^{1,p}(\om,\mathbb R^3),
\end{array}\right.
\end{equation*}
where $H^1_{\dv}(\om,\mathbb R^3)$ denotes the space of divergence-free $H^1(\om,\R^3)$ vector fields.
 \MMM Such a result improves the one in \cite{MP2}, as it allows to obtain the characterization of the limit energy even 
  without the assumption $\mathcal S^0_{\mathcal L}\equiv\{\mathbf I\}$. \KKK
 It
 also generalizes a recent result of Jesenko and Schmidt \cite{JS} and  reduces to it  when  $\mathcal S^0_{\mathcal L}\equiv SO(3)$.
 We will provide 
 the same statement 
  for the compressible case  in Theorem \ref{mainth1comp},  thus obtaining an analogous of  \cite[Theorem 5.3]{MM} for the case of the $p$-growth assumption \eqref{pgr}  (see also Remark \ref{maormora} below). 
  \KKK On top of that, we will show in Theorem \ref{mainth3} that there are configurations and external loads such that the strict inequalities  \eqref{new} hold, the minimization problems being cast on $W^{1,p}(\om,\mathbb R^3)$. 
  We will end our analysis by remarking that \eqref{heuristic} might be true even if  \eqref{c}-\eqref{wcomp} hold and  $\mathcal S^0_{\mathcal L}$ is not reduced to the identity matrix: indeed, it is always possible to rotate the external forces in such a way that \eqref{heuristic} holds for the problem with rotated forces, see Theorem \ref{rotheorem}.\KKK  
 \MMM
 
 \medskip 
 
 \KKK
 Let us finally mention that several other results about variational linearization of finite elasticity, including Dirichlet problems, incompressibility constraints or even theories for multiwell potentials
 are found in \cite{ABK, ADMDS, ADMLP, DMPN, MP,  S}.

 \subsection*{Plan of the paper} 
 In Section \ref{sectmain} we introduce the assumptions of the theory and state the main results. Section \ref{prel} collects some preliminary results. In Section \ref{proofsection} we provide the proof of the variational convergence results. Eventually, Section \ref{counterexamples} delivers the main example with a limiting energy that is below the minimal value of the standard linearized elasticity functional.

\section{Main Results}\label{sectmain}

We introduce the setting for compressible and incompressible elasticity, then we state the main results.
In the following, the reference configuration $\Omega$ is always assumed to be a bounded open connected Lipschitz set in $\mathbb R^3$.

As basic notation,  $\mathbb R^{3\times 3}$ is the set of $3\times 3$ real matrices, endowed with the Euclidean norm $|\mathbf F|=\sqrt{\mathbf F^T\mathbf F}$.
$\mathbb R^{3\times 3}_{\rm sym}$  (resp. $\mathbb R^{3\times 3}_{\rm skew})$ denotes  the subset of symmetric (resp. skew-symmetric) matrices. For every $\mathbf F\in \mathbb R^{3\times 3}$ we define ${\rm sym\,}\mathbf F:=\frac{1}{2}(\mathbf F+\mathbf F^T)$
and  ${\rm skew\,}\mathbf F:=\frac{1}{2}(\mathbf F-\mathbf F^T)$. 
By $SO(3)$ we  denote the special orthogonal group and for every $\mathbf R\in SO(3)$ there exist $\vartheta\in \mathbb R$ and $\mathbf W\in \mathbb R^{3\times 3}_{\mathrm{skew}},$ such that  $|\mathbf W|^{2}=|\mathbf W^2|^2=2$ and such that 
the following  Euler-Rodrigues representation formula holds
\begin{equation}\label{eurod}
\mathbf R\,=\,\mathbf I\,+\,\sin\vartheta \,\mathbf W\,+\,(1-\cos\vartheta)\,\mathbf W^{2}.
\end{equation}


\subsection*{Assumptions on the elastic energy density} 

We let
  $\mathcal W : \om \times \mathbb R^{3 \times 3} \to [0, +\infty ]$ be  ${\mathcal L}^3\! \times\! {\mathcal B}^{9} $- measurable 
 satisfying the following assumptions, see also \cite{ADMDS,MP}: 
%
%
\beeq \lab{framind}\tag{$\bb{\mathcal W1}$} \W(\x, \mathbf R\mathbf F)=\W(\x, \mathbf F) \qquad \forall \, \mathbf\! \mathbf R\!\in\! SO(3) \quad \forall\, \mathbf F\in \mathbb R^{3 \times 3},\qquad \mbox{for a.e. $\x\in\Omega$},
\eneq
\beeq \lab{Z1}\tag{$\bb{\mathcal W2}$}
\min \mathcal W=\KKK	\W^I(\x,\mathbf I)=0 \quad \mbox{for a.e. $\x\in\Omega$}.
\eneq
Concerning the regularity of $\mathcal W$, we assume that there exist an open neighborhood $\mathcal U$ of $SO(3)$ in $\R^{3\times3}$,  an increasing function $\omega:\mathbb R_+\to\mathbb R$ satisfying $\lim_{t\to0^+}\omega(t)=0$ and a constant $K>0$
such that for a.e. $\x\in\om$
\beeq\begin{array}{ll}\lab{reg}\tag{$\bb{\mathcal W3}$} &   
\vspace{0,1cm}
\mathcal W(\x,\cdot)\in C^{2}(\mathcal U),\;\;\;
 \vspace{0,1cm}
  |D^2 \mathcal W(\x,\mathbf I)|\le K \;\;\hbox{and}\\
& 
 |D^2\W(\x,\mathbf F)-D^2\W(\x,\mathbf G)|\le\omega(|\mathbf F-\mathbf G|)\quad\forall\; \mathbf F,\mathbf G\in\mathcal U.
\end{array}
\eneq
We assume in addition the following growth property from below: there exist $C>0$ and $p\in(1,2]$ such that for a.e. $\x\in\Omega$
\beeq \lab{coerc}\tag{$\bb{\mathcal W4}$}
\begin{array}{ll}
\W(\x,\mathbf F)\ge  C \ g_{p}(d(\mathbf F, SO(3)))\qquad
 \forall\, \mathbf F\in \mathbb R^{3 \times 3},
\end{array}
\eneq
where  $g_p:[0,+\infty)\to\mathbb R$ is the strictly convex function defined by \eqref{gp}.
We notice that a standard application of the H\"older inequality shows that for every $\eta \in L^p(\om)$ and every $h\in (0,1)$
 \begin{equation}\lab{propgp}\begin{aligned}
\displaystyle h^{-2}\int_\om g_p(h|\eta|)\,dt&\ge \int_{|\eta|\le h^{-1}}|\eta|^2\,dt+h^{p-2}\int_{|\eta|\ge h^{-1}}|\eta|^p\,dt\\
&\displaystyle\ge \frac{2}{p} \int_{|\eta|\le h^{-1}}|\eta|^p\,dt+h^{p-2}\int_{|\eta|\ge h^{-1}}|\eta|^p\,dt-\frac{2-p}{p}\,|\om|\\
&\displaystyle\ge \int_{\om}|\eta|^p\,dt-\frac{2-p}{p}\,|\om|
\end{aligned}
\end{equation}

In order to consider incompressible elasticity models, starting from a function $\mathcal W$ as above we also introduce the incompressible strain energy density by letting, for a.e. $\x\in\Omega$,
\begin{equation*}
{\mathcal W}^I (\x, \mathbf F):=\left\{\begin{array}{ll} \mathcal W(\x,\mathbf F) \qquad &\hbox{if}\ \ \det\mathbf F=1\\
 +\infty\ &\hbox{otherwise.}
\end{array}\right.
\end{equation*}

\KKK


%



\subsection*{Assumptions on the external forces}
We introduce  a  body force field $\mathbf f\in L^{\frac{3p}{4p-3}}(\Omega,\R^3)$ and a surface force field $\mathbf g\in L^{\frac{2p}{3p-3}}(\partial\Omega,\R^3)$, where $p$ is such that \eqref{coerc} holds. From here on,  $\mathbf f$ and $\mathbf g$ will  always be understood to satisfy such summability assumptions. The load functional is the following linear functional \begin{equation}\label{external}
\mathcal L(\v):=\int_\om\mathbf f\cdot\v\,d\x+\int_{\partial\om}\mathbf g\cdot\v\,d\mathcal H^2(\x),\qquad \v\in W^{1,p}(\om,\mathbb R^3).
\end{equation}
We note that since $\Omega$ is a bounded Lipschitz domain, the Sobolev embedding $W^{1,p}(\Omega,\mathbb R^3)\hookrightarrow L^{\frac{3p}{3-p}}(\Omega,\mathbb R^3)$ and the Sobolev trace embedding $W^{1,p}(\Omega,\mathbb R^3)\hookrightarrow L^{\frac{2p}{3-p}}(\partial\Omega,\mathbb R^3)$ imply that $\mathcal L$ is a bounded functional over $W^{1,p}(\Omega,\mathbb R^3)$.

 We assume that external  loads have null resultant 
 \begin{equation}\lab{globalequi}\tag{$\bb{\mathcal L1}$}
\mathcal L(\mathbf c)=0\qquad \forall\, \mathbf c\in \mathbb R^3
\end{equation}
 and that they satisfy the following  weak compatibility condition\KKK 
\begin{equation}\lab{comp}\tag{$\bb{\mathcal L2}$}
\mathcal L((\mathbf R-\mathbf I)\x)\le 0\qquad \forall\,\mathbf R\in SO(3).\end{equation}
A crucial object in our results is the rotation kernel $\mathcal S_{\mathcal L}^0$ associated to a functional $\mathcal L$ satisfying the above assumptions, which is the set defined by \eqref{S0}.
Such a kernel includes at least the identity matrix and represents the set of rotations that realize equality in \eqref{comp}.

\begin{remark}\label{rem21}\rm Thanks to the the Euler-Rodrigues representation formula for rotations \eqref{eurod}, it is readily seen that \eqref{comp} may be rewritten as
\begin{equation*}\lab{comp2}
h_{\mathbf W}(\theta):=\mathcal L(\mathbf W\x)\sin\theta+(1-\cos\theta)\mathcal L(\mathbf W^2\x)\le 0
\end{equation*}
for every $\theta\in [0,2\pi]$ and for every $\mathbf W\in \mathbb R^{3\times 3}_{\mathrm{skew}}$ with $|\mathbf W|=2$.
Since $h_{\mathbf W}(0)=h_{\mathbf W}(2\pi)=0$, then 
$$0\le h_{\mathbf W}'(2\pi)=\mathcal L(\mathbf W\x)=h_{\mathbf W}'(0)\le 0,$$
that is, by linearity of $\mathcal L$,
\begin{equation*}
\mathcal L(\mathbf W\x)=0\ \quad \forall\ \mathbf W\in \mathbb R^{3\times 3}_{\mathrm{skew}},
\end{equation*}
and so  if \eqref{comp} holds then  external loads have null momentum. Therefore, \eqref{comp} is equivalent to
\begin{equation}\label{comp3}
\mathcal L(\mathbf W\x)=0,\ \ \mathcal L(\mathbf W^2 \x)\le 0\qquad \forall\ \mathbf W\in \mathbb R^{3\times 3}_{\mathrm{skew}},
\end{equation}
and we mention that formulation \eqref{comp3} of the compatibility condition  (with strict inequality) is the one appearing in \cite{MPTJOTA, MPTARMA, MP2}.
 On the other hand it is worth noting that the null momentum condition does not imply the second relation in \eqref{comp3}. 
Indeed, let  $\mathbf f(\x)= - \x,\ \mathbf g\equiv 0$ and let $\om$ be the open unit ball in $\mathbb R^3$. Then $$\mathcal L(\mathbf W^2\x)=\int_\om |\mathbf W\x|^2\,d\x> 0$$
for every  $\mathbf W\in \mathbb R^{3\times 3}_{\mathrm{skew}}, \ \mathbf W\not\equiv 0$, despite that external loads has null resultant and null momentum. 
\end{remark}

\begin{remark}\label{subgroup} \rm The characterization \eqref{comp3} of \eqref{comp} and Euler-Rodrigues formula  entail
\begin{equation*}
 \mathcal S^{0}_{\mathcal L}=\left\{ \mathbf R\in SO(3): \mathcal L( (\mathbf R- \mathbf I)\x)=0\right\}=\left\{ e^{\mathbf W} : \mathbf W\in {\mathcal X}^0_{\mathcal L}\right\}
\end{equation*}
where
$${\mathcal X}^0_{\mathcal L}:=\left\{\mathbf W\in \mathbb R^{3\times 3}_{\mathrm{skew}}:\; \mathcal L(\mathbf W\x)= \mathcal L(\mathbf W^2 \x)=0\right\}.$$
Therefore,
we have $\mathbf R\in\mathcal S_{\mathcal L}^0\Rightarrow \mathbf R^T\in\mathcal S_{\mathcal L}^0$, because $$\mathcal L((\mathbf R^T-\mathbf I)\x)=\mathcal L((\mathbf R^T-\mathbf R)\x)+\mathcal L((\mathbf R-\mathbf I)\x)=0$$ holds true since $\mathbf R^T-\mathbf R$ is skewsymmetric.
Moreover, if $\mathbf W_i\in {\mathcal X}^0_{\mathcal L},\ i=1,2$, then by \eqref{comp}
\begin{equation*}\begin{aligned}0&\ge \mathcal L((\mathbf W_1\pm\mathbf W_2)^{2}\x)=\mathcal L(\mathbf W^2_1 \x)+\mathcal L(\mathbf W^2_2 \x)\pm \mathcal L((\mathbf W_1\mathbf W_2+\mathbf W_2\mathbf W_1)\x)\\&=\pm \mathcal L((\mathbf W_1\mathbf W_2+\mathbf W_2\mathbf W_1)\x),\end{aligned}\end{equation*}
that is, $\mathbf W_1\pm\mathbf W_2\in {\mathcal X}^0_{\mathcal L}$, hence $e^{\mathbf W_1\pm \mathbf W_2}\in  \mathcal S^{0}_{\mathcal L}$. By recalling that $\mathbf I\in \mathcal S^{0}_{\mathcal L}$, we conclude that $\mathcal S^{0}_{\mathcal L}$ is a subgroup of $SO(3)$. We refer to \cite{MM} for a more detailed characterization of the set $\mathcal S_{\mathcal L}^0$.
\end{remark}
\KKK



\subsection*{Energy functionals}

  The rescaled finite elasticity functionals  $\mathcal G_h: W^{1,p}(\Omega,\R^3)\to \R\cup\{+\infty\} $ are defined by \eqref{un}
%
and the limit energy functional $\mathcal G:W^{1,p}(\Omega,\mathbb R^3)\to\mathbb R\cup\{+\infty\}$ is defined as
\begin{equation*}\lab{elfunc}
\displaystyle {\mathcal G}(\u):=\left\{\begin{array}{ll}\displaystyle \int_\om \mathcal Q(\x, \mathbb E(\u))\,d\x-\mathcal L(\u)-\max_{\mathbf R\in \mathcal S^0_{\mathcal L}} \mathcal L((\mathbf R-\mathbf I)\u)\quad &\hbox{if} \ \u\in H^1(\om,\mathbb R^3)\\
&\\
 \ \!\!+\infty\quad &\hbox{otherwise in} \ W^{1,p}(\om,\mathbb R^3),
\end{array}\right.
\end{equation*}
where    ${\mathcal Q} (\x, \mathbf F):=\tfrac12\,\mathbf F^T\,D^2\mathcal W(\x,\mathbf I)\,\mathbf F$.
By introducing the standard functional of linearized elasticity $\mathcal E:W^{1,p}(\Omega,\mathbb R^3)\to\mathbb R\cup\{+\infty\}$, namely
\begin{equation*}
\displaystyle {\mathcal E}(\u):=\left\{\begin{array}{ll}\displaystyle \int_\om \mathcal Q(\x, \mathbb E(\u))\,d\x-\mathcal L(\u)\quad &\hbox{if} \ \u\in H^1(\om,\mathbb R^3)\\
&\\
 \ \!\!+\infty\quad &\hbox{otherwise in} \ W^{1,p}(\om,\mathbb R^3),
\end{array}\right.
\end{equation*}
we immediately see that $\mathcal G\le \mathcal E$, since $\mathbf I\in\mathcal S^0_{\mathcal L}$ and $\mathcal L(\mathbf 0)=0$.
\MMM It is well-known that $\mathcal E$ admits a unique minimizer up to infinitesimal rigid displacements (i.e., up to the addition of a displacements field $\v$ such that $\mathbb E(\v)=0$). Since the optimization problem in the definition of $\mathcal G$ is among rotations in $\mathcal S_{\mathcal L}^0$, it is not difficult to check that $\mathcal G$ is invariant under the addition of infinitesimal rigid displacements, i.e., $\mathcal G(\u+\v)=\mathcal G(\u)$ whenever $\mathbb E(\v)\equiv0$. On the other hand, in general minimizers of $\mathcal G$ are not unique up to infinitesimal rigid displacements (see Proposition \ref{pro52} later on). 

 \KKK 


\bigskip

When considering incompressible elasticity,
 the functional  $\mathcal G_h^{I}: W^{1,p}(\Omega,\R^3)\to \R\cup\{+\infty\} $, representing the scaled total energy, is defined by
\begin{equation*}
\label{nonlinear}
\displaystyle \mathcal G_h^{I}(\mathbf y):=h^{-2}\int_\om\mathcal W^I(\x,\nabla\mathbf y)\,d\x
-h^{-1}\mathcal L(\mathbf y-\mathbf i),
\end{equation*}
%
while the limit  functional $\mathcal G^I:W^{1,p}(\Omega,\mathbb R^3)\to\mathbb R\cup\{+\infty\}$ is defined by
\begin{equation*}\lab{elfunc2}
\displaystyle {\mathcal G}^I(\u):=\left\{\begin{array}{ll}\displaystyle \int_\om \mathcal Q^I(\x, \mathbb E(\u))\,d\x-\mathcal L(\u)-\max_{\mathbf R\in \mathcal S^0_{\mathcal L}} \mathcal L((\mathbf R-\mathbf I)\u)\quad &\hbox{if} \ \u\in H^1(\om,\mathbb R^3)\\
&\\
 \ \!\!+\infty\quad &\hbox{otherwise in} \ W^{1,p}(\om,\mathbb R^3),
\end{array}\right.
\end{equation*}
where $\mathcal Q^I$ is defined by \eqref{QI}. 
We also introduce the functional of incompressible linearized elasticity
$\mathcal E^I:W^{1,p}(\Omega,\mathbb R^3)\to\mathbb R\cup\{+\infty\}$, namely
\begin{equation*}
\displaystyle {\mathcal E}^I(\u):=\left\{\begin{array}{ll}\displaystyle \int_\om \mathcal Q^I(\x, \mathbb E(\u))\,d\x-\mathcal L(\u)\quad &\hbox{if} \ \u\in H^1(\om,\mathbb R^3)\\
&\\
 \ \!\!+\infty\quad &\hbox{otherwise in} \ W^{1,p}(\om,\mathbb R^3),
\end{array}\right.
\end{equation*}
and again $\mathcal G^I\le \mathcal E^I$. \MMM Functional $\mathcal E^I$ admits a unique minimizer up to infinitesimal rigid displacements.\KKK

\bigskip

Before moving to the statement of the main results, we introduce a couple of definitions.
For every $\mathbf y\in W^{1,p}(\om,\mathbb R^3)$ we define 
\begin{equation*}\lab{roty}
\mathcal A_p(\mathbf y):= \argmin\left\{ \int_\om g_p(|\nabla\mathbf y-\mathbf R|)\,d\x: \mathbf R\in SO(3)\right\}.
\end{equation*}
Moreover,
 we may combine this definition with
the rigidity inequality by Friesecke, James and M\"uller \cite{FJM0}, in the general form appearing in  \cite{FJM,ADMDS}, to get the following estimate.
There exists a constant $C_p=C_{p}(\om) >0$ such that for every $\yy\in W^{1,p}(\om,\mathbb R^{3})$ and every $\mathbf R\in \mathcal A_p(\mathbf y)$
\begin{equation}\label{muller}
\int_{\om}g_{p}(|\nabla\yy-\mathbf R|)\,d\x\le C_{p}\int_{\om}g_{p}(d(\nabla \yy, SO(3)))\,d\x,
\end{equation}\KKK
where $d(\mathbf F,SO(3)):=\inf\{|\mathbf F-\mathbf R|:\mathbf R\in SO(3)\}$.

Moreover, we introduce the following
\begin{definition}
 Given a vanishing sequence $(h_j)_{j\in\mathbb N}\subset(0,1)$, we say that $(\mathbf y_j)_{j\in\mathbb N}\subset W^{1,p}(\Omega,\mathbb R^3)$ is a sequence of quasi-minimizers of $\mathcal G_{h_j}$ if
 \[
 \lim_{j\to+\infty}\left(\mathcal G_{h_j}(\mathbf y_j)-\inf_{W^{1,p}(\Omega,\mathbb R^3)}\mathcal G_{h_j}\right)=0.
 \]
 Sequences of quasi-minimizers of $\mathcal G^I_{h_j}$ are defined in the same way.
 \end{definition}

\subsection*{Convergence results}
We  are ready for the statement of the convergence result in the compressible case.
\begin{theorem}
\label{mainth1comp}
Assume \eqref{globalequi}, \eqref{comp},  \eqref{framind}, \eqref{Z1}, \eqref{reg},    
\eqref{coerc}. 
Let $(h_j)_{j\in\mathbb N}\subset(0,1)$  be a vanishing sequence.
Then we have
\[
\inf_{W^{1,p}(\om,\mathbb R^3)}\mathcal G_{h_{j}} \in \mathbb R \qquad\mbox{for any $j\in\mathbb N$}. \]
Moreover, if $(\mathbf y_j)_{j\in\mathbb N}\subset W^{1,p}(\om,\mathbb R^3)$ is a sequence of quasi-minimizers of $\mathcal G_{h_j}$, 
 and if $\mathbf R_j\in\mathcal A_p(\mathbf y_j)$ for any $j\in\mathbb N$, \KKK then by defining
 \begin{equation*}
\u_j(\x):= h_j^{-1}(\mathbf R_{j}^T\mathbf y(\x)-\x)
\end{equation*}\KKK
there is a (not relabeled) subsequence such that
\[ \nabla\u_{j}\wconv \nabla\u_* \ \hbox{weakly in}\ L^{p}(\Omega,\R^{3\times3})\qquad\mbox{as $j\to+\infty$},
\]
where $\u_*\in H^1(\om,\R^3)$ is a minimizer of $\mathcal G$ over $W^{1,p}(\om,\mathbb R^3)$,
and
\begin{equation*}   \mathcal G_{h_{j}}(\mathbf y_{j})\to \mathcal G(\u_*),\quad\qquad
\inf_{W^{1,p}(\om,\R^3)}\mathcal G_{h_{j}}\to \min_{W^{1,p}(\om,\R^3)}\mathcal G\quad\qquad\mbox{as $j\to+\infty$}.
\end{equation*}
\end{theorem}

The same statement holds in the incompressible case
 
\begin{theorem}
\label{mainth1}
Assume that $\partial\Omega$ has a finite number of connected components.
Assume \eqref{globalequi}, \eqref{comp},  \eqref{framind}, \eqref{Z1}, \eqref{reg},    
\eqref{coerc}. 
Let $(h_j)_{j\in\mathbb N}\subset(0,1)$  be a vanishing sequence.
Then we have
\beeq\lab{convmin}
\inf_{W^{1,p}(\om,\mathbb R^3)}\mathcal G^{I}_{h_{j}} \in \mathbb R\qquad \mbox{for any $j\in\mathbb N$}. \eneq
Moreover, if $(\mathbf y_j)_{j\in\mathbb N}\subset W^{1,p}(\om,\mathbb R^3)$ is a sequence  of quasi-minimizers of $\mathcal G^I_{h_j}$, and if $\mathbf R_j\in\mathcal A_p(\mathbf y_j)$ for any $j\in\mathbb N$,
then  by defining 
 \begin{equation*}
\u_j(\x):= h_j^{-1}(\mathbf R_{j}^T\mathbf y(\x)-\x)
\end{equation*}
there is a (not relabeled) subsequence such that
\[ \nabla\u_{j}\wconv \nabla\u_* \ \hbox{weakly in}\ L^{p}(\Omega,\R^{3\times3})\qquad\mbox{as $j\to+\infty$},
\]
where $\u_*\in H^1(\om,\R^3)$ is a minimizer of $\mathcal G^I$ over $W^{1,p}(\om,\mathbb R^3)$,
and
\begin{equation*}   \mathcal G^{I}_{h_{j}}(\mathbf y_{j})\to \mathcal G^{I}(\u_*),\qquad\quad
\inf_{W^{1,p}(\om,\R^3)}\mathcal G^{I}_{h_{j}}\to \min_{W^{1,p}(\om,\R^3)}\mathcal G^{I}\qquad\quad\mbox{as $j\to+\infty$}.
\end{equation*}
\end{theorem}
\begin{remark}\lab{-infty} \rm
Inequality in \eqref{comp} can never be reversed. Indeed let us assume \eqref{globalequi},  \eqref{framind}, \eqref{Z1}, \eqref{reg},    
\eqref{coerc} and  that there exists $\mathbf R_*\in SO(3)$ such that
\begin{equation}\lab{reversecomp}
\mathcal L((\mathbf R^*-\mathbf I)\x)> 0.
\end{equation}
Then it is readily seen that by setting $\mathbf y_j^*(\x)= \mathbf R^*\x$ we get
\begin{equation*}\lab{idemenergy}
\mathcal G_{h_{j}}(\mathbf y_j^*)=
-h_j^{-1}\mathcal L((\mathbf R^*-\mathbf I)\x)\to -\infty\qquad\mbox{as $j\to+\infty$}.
\end{equation*}
For instance,  we notice that if the body is subject to a uniform boundary compressive  force field  then \eqref{reversecomp} occurs for every $\mathbf R\in SO(3),\ \mathbf R\neq \mathbf I$. Indeed, if $\n$ denotes the outer unit normal vector to $\partial\om$, and we choose
$\mathbf g=\lambda\n$ with $\lambda< 0$ and $\mathbf f\equiv 0$, then
\[
\int_{\partial\Omega}\mathbf g\cdot(\mathbf R-\mathbf I) \x\ d\H^{2}(\x)\ =\ \lambda\,(\Tr(\mathbf R-\mathbf I))\,\vert\Omega\vert
\ >\ 0
\qquad \forall \
\mathbf R\!\in\! SO(3),\  \mathbf R\neq \mathbf I.\]
\end{remark}

\subsection*{The gap with linear elasticity}
 By summarizing, if \eqref{globalequi}, \eqref{comp},  \eqref{framind}, \eqref{Z1}, \eqref{reg},    
\eqref{coerc} are satisfied and $\mathcal S^0_{\mathcal L}\equiv\{\mathbf I\}$ then  functionals $\mathcal G^I$ and $\mathcal G$ are the classical functionals of linear incompressible and compressible elasticity, respectively.
 On the other hand if \eqref{reversecomp} occurs then by Remark \ref{-infty} no convergence result is possible. In all other cases, although a full description of the limit functionals is given by means of  $\mathcal G$ and $\mathcal G^I$, \KKK it is not a priori \KKK clear if their  minimal values coincide with the minimal values \KKK of linear (compressible or incompressible) elasticity.
In order to complete the picture, we will show that there exist configurations and external forces satisfying the assumptions of Theorem {\rm \ref{mainth1comp}} and {\rm Theorem \ref{mainth1}} (and such that $\mathcal S^0_{\mathcal L}$ is not reduced to the identity matrix) for which 
\begin{equation}\label{ineq1}\min_{W^{1,p}(\om,\R^3)}\mathcal G< \min_{W^{1,p}(\om,\R^3)}\mathcal E
\end{equation}
and
\begin{equation}\label{ineq2}
\min_{W^{1,p}(\om,\R^3)}\mathcal G^I <\min_{W^{1,p}(\om,\R^3)}\mathcal E^I,
\end{equation}
Therefore, in such case,  given any vanishing sequence $(h_j)_{j\in\mathbb N}\subset(0,1)$ and any sequence $(\mathbf y_j)_{j\in\mathbb N}\subset W^{1,p}(\om,\R^3)$ of quasi-minimizers of $\mathcal G_{h_j}$ (resp. $\mathcal G^I_{h_j}$) there is no subsequence such 
that $\mathcal G_{h_j}(\mathbf y_j)\to\min\mathcal E$ (resp. $\mathcal G^I_{h_j}(\mathbf y_j)\to\min\mathcal E^I$). 
This shows that the minimal value of the usual functional of linearized elasticity $\mathcal E$ (resp. $\mathcal E^I$) \MMM need not be the correct  approximation of $\inf\mathcal G_h$  (resp. $\inf\mathcal G^I_h$) in the regime of small $h$. \KKK

\medskip

We are going to analyze in detail an example of validity of \eqref{ineq1} and \eqref{ineq2}. The setting for such an example is the following.  
 We assume that $\Omega$ is a cylinder: \begin{equation}\label{cyl2}\Omega:=\{(x,y,z)\in\mathbb R^3: x^2+y^2<1,\,0<z<1\}.\end{equation} 
 We consider a particular choice for functional $\mathcal L$ from \eqref{external},
  letting $\mathbf g\equiv\mathbf 0$ and  letting
  $\mathbf f\in L^2(\Omega,\mathbb R^3)$ have a specific form. Namely, we let \begin{equation}\label{ef2}
  \mathcal L(\u)=\int_\Omega \mathbf f\cdot \u\qquad\mbox{with}\qquad
  \mathbf f(x,y,z):=(\varphi_{x}(x,y), \varphi_y(x,y),\psi(z)),\end{equation}
 where $\varphi$ and $\psi$ satisfy  
 the following restrictions (where $B$ denotes the unit ball centered at the origin in the $xy$ plane and $\Delta$ denotes the Laplacian in the $x,y$ variables):
 \begin{equation}\label{effe12}\tag{${\mathbf f\bb1}$}\begin{aligned}
  &\mbox{$\varphi\in C^2(\overline B)$   is radial,  there exists $(x,y)\in B$ such that $\Delta\varphi(x,y)\neq 0$,}\\[-5pt]
  & \mbox{$\phi(1)=\phi'(1)= \int_0^1 r^2\phi'(r)\,dr=0$,} \mbox{ where $\phi$ denotes the radial profile of $\varphi$,}\end{aligned}\end{equation}
 \begin{equation}\label{effe22}\tag{${\mathbf f\bb2}$}
 \mbox{$\psi\in C^0([0,1]),\quad
 \int_0^1\psi(z)\,dz=0, \quad\int_0^1z\psi(z)\,dz\ge 0$}.
 \end{equation}
 Condition \eqref{effe12} can be satisfied by choosing for instance $\varphi$ to be a suitable polynomial in the radial variable, like $\phi(r)=4r^6-9r^4+6r^2-1$.  As we will discuss in Section \ref{counterexamples}, with such choices of $\Omega$ and $\mathcal L$ the conditions \eqref{globalequi} and \eqref{comp} are satisfied and moreover $\{\mathbf I\}\subsetneq \mathcal S_{\mathcal L}^0$. We also have  $\mathcal S_{\mathcal L}^0\subsetneq SO(3)$ if $\int_0^1z\psi(z)\,dz> 0$ and $\mathcal S_{\mathcal L}^0\equiv SO(3)$ if $\int_0^1z\psi(z)\,dz=0.$ 

 Concerning the strain energy functionals, we choose 
 any function $\mathcal W(\x,\mathbf F)=\mathcal W (\mathbf F)$ satisfying assumptions  \eqref{framind}, \eqref{Z1}, \eqref{reg}, \eqref{coerc} and being such that 
 \begin{equation}
 \label{Venant3} 
 \frac12\,\mathbf F^T\, D^2\mathcal W(\mathbf I)\,\mathbf F=4\,|\mathbf F|^2\qquad\forall \ \mathbf F\in \R^{3\times3}.
 \end{equation}
  Accordingly we let $\mathcal W^I(\x,\mathbf F)=\mathcal W^I(\mathbf F)$ be equal to $\mathcal W(\mathbf F)$ if $\det\mathbf F=1$ and equal to $+\infty$ otherwise.
 An example is the homogeneous Kirchoff - Saint-Venant energy, obtained by setting
 \begin{equation*}\label{Venant2}\mathcal W(\x,\mathbf F)=\mathcal W(\mathbf F)=|\mathbf F^T\mathbf F-\mathbf I|^2.\end{equation*}
  We have the following

 \begin{theorem}\label{mainth3}
  Assume \eqref{cyl2}, \eqref{ef2}, \eqref{effe12}, \eqref{effe22}, \eqref{framind}, \eqref{Z1}, \eqref{reg}, \eqref{coerc} and \eqref{Venant3}. Then, the assumptions of {\rm Theorem \ref{mainth1comp}} and of {\rm Theorem \ref{mainth1}} are satisfied,
 \eqref{ineq1} holds true,
and if  $\|\psi\|_{L^2(0,1)}$ is small enough \eqref{ineq2} holds true as well. \KKK
 \end{theorem}

\begin{remark}\rm 
In the assumptions  of   Theorem \ref{mainth3}, let us consider the rescaled displacement fields $\v_j$ and the generalized rescaled displacement fields $\u_j$ associated to a sequence of quasi-minimizers $(\mathbf y_j)\subset W^{1,p}(\om,\mathbb R^3)$ of $\mathcal G_{h_j}$ (or $\mathcal G^I_{h_j}$). Since $\v_j(\x)=h_j^{-1}(\mathbf y_j(\x)-\x)$ and $\u_j(\x)=h_j^{-1}(\mathbf R_j^T\mathbf y_j(\x)-\x)$,  where $\mathbf R_j\in\mathcal A_p(\mathbf y_j)$,  \KKK we have
\begin{equation}\label{uv}
\mathbb E(\v_j)=\frac{\mathbf R_j^T+\mathbf R_j-2\mathbf I}{2h_j}-\mathrm{sym}(\mathbf R_j\nabla\u_j).
\end{equation}
Along a suitable subsequence we have $\nabla \u_j\to \nabla \u$ weakly in $L^p(\om,\mathbb R^{3\times3})$ and $\mathbf R_j\to\mathbf R_*\in \mathcal S_{\mathcal L}^0$ as we will show in Section \ref{proofsection}.
However, along the same sequence, $\mathbb E(\v_j)$ is unbounded in $L^p(\om,\mathbb R^{3\times3})$, otherwise the results of \cite{MPTARMA} and \cite{MP2} would entail convergence to the minimal value of the standard linearized elasticity functional, in contrast with Theorem \ref{mainth3}.  In fact, Theorem \ref{mainth3} shows that $\mathbf R_*\neq\mathbf I$: the optimal rotation at a minimizer $\u$ of $\mathcal G$ (or $\mathcal G^I$) is the limit of the rotations $\mathbf R_j$, it is not the identity matrix and then \eqref{uv} confirms that $\mathbb E(\v_j)$ is unbounded in $L^p(\om,\mathbb R^{3\times3})$. \KKK
\end{remark}

\subsection*{Rotated external forces with no gap}
 Going back to the general setting,
 it is clear that a necessary condition for the validity of \eqref{ineq1} and \eqref{ineq2} is that $\mathcal S_{\mathcal L}^0$ is not reduced to the identity matrix. 
However, such a condition is not sufficient for the presence of a gap with linear elasticity. This is immediately seen by choosing $\mathbf f\equiv \mathbf 0$ and $\mathbf g\equiv \mathbf 0$, in which case $\mathcal S^0_{\mathcal L}\equiv SO(3)$ but of course
all the four minima appearing in \eqref{ineq1}-\eqref{ineq2} are equal to zero. A much more general result holds, showing that the appearance of the gap is strongly influenced by the initial choice of the external loads. 
Before providing such result, we  introduce some further notation.
 For every $\mathbf R\in SO(3)$, let
 \begin{equation*}
\displaystyle {\mathcal G}_{h,\mathbf R}(\mathbf y):= h^{-2}\int_\om \mathcal W(\x,\nabla\mathbf y)\,d\x- h^{-1}\mathcal L_{\mathbf R}(\mathbf y(\x)-\x),\qquad \mathbf y\in W^{1,p}(\om,\mathbb R^3),
\end{equation*}
\begin{equation*}\label{rotelfunc}
\displaystyle {\mathcal G}_{\mathbf R}(\v):=\left\{\begin{array}{ll}\displaystyle \int_\om \mathcal Q(\x, \mathbb E(\v))\,d\x-\max_{\hat{\mathbf  R}\in\mathcal S_{\mathcal L_{\mathbf R}}^0}\mathcal L_{\mathbf R}(\hat{\mathbf R}\v) &\hbox{if} \ \v\in H^1(\om,\mathbb R^3)\\
&\\
 \ \!\!+\infty\quad &\hbox{otherwise in} \ W^{1,p}(\om,\mathbb R^3),
\end{array}\right.
\end{equation*}
where $\mathcal L_{\mathbf R}: W^{1,p}(\om,\mathbb R^3)\to\mathbb R$ is defined as
\begin{equation}\label{elleerre}
\mathcal L_{\mathbf R}(\v):=\mathcal L(\mathbf R\v)=\int_\om\mathbf R^T\mathbf f\cdot\v\,d\x-\int_{\partial\om}\mathbf R^T\mathbf f\cdot\v\,d\mathcal H^2(\x),
\end{equation}
so that $\mathcal L_{\mathbf R}$ is just the usual load functional associated to the external forces $\mathbf R^T\mathbf f,\mathbf R^T\mathbf g$.
Similarly, $\mathcal E_{\mathbf R}$ is defined by replacing $\mathcal L$ with $\mathcal L_{\mathbf R}$ in the definition  of $\mathcal E$.
The corresponding functionals ${\mathcal G}^I_{h,\mathbf R},\mathcal G^I_{\mathbf R},\mathcal E^I_{\mathbf R}$ of incompressible elasticity are also defined  by replacing $\mathcal L$ with $\mathcal L_{\mathbf R}$ in the definition  of $\mathcal G^I_h, \mathcal G^I$ and $\mathcal E^I$, respectively.
\begin{remark} \rm
Since $\mathcal L$ satisfies \eqref{globalequi} and \eqref{comp}, it is clear that given $\mathbf R\in\mathcal S_{\mathcal L}^0$ there holds $\mathcal L_{\mathbf R}(\mathbf c)=0$ for every $\mathbf c\in\mathbb R^3$ and 
  \begin{equation}\label{eqi}
  \mathcal L_{\mathbf R}((\mathbf S-\mathbf I)\x)=\mathcal L((\mathbf R\mathbf S-\mathbf I)\x)-\mathcal L((\mathbf R-\mathbf I)\x)=\mathcal L((\mathbf R\mathbf S-\mathbf I)\x)\le 0
  \end{equation}
  for every $\mathbf S\in SO(3)$, thus showing that $\mathcal L_{\mathbf R}$ satisfies   \eqref{globalequi} and \eqref{comp} as well.
  Therefore, by Theorem \ref{mainth1comp} and Theorem \ref{mainth1}, functionals ${\mathcal G}_{\mathbf R}$ and $\mathcal G^I_{\mathbf R}$ can be viewed as the limit of functionals ${\mathcal G}_{h,\mathbf R}$ and ${\mathcal G}^I_{h,\mathbf R}$ respectively.
  \end{remark}
  
 \begin{theorem}\label{rotheorem} Assume \eqref{globalequi}, \eqref{comp},  \eqref{framind}, \eqref{Z1}, \eqref{reg},    
\eqref{coerc}. If $\mathbf R\in \mathcal S_{\mathcal L}^0$, the rotated load functional $\mathcal L_{\mathbf R}$  still satisfies $\eqref{globalequi}$ and $\eqref{comp}$,  and $\mathcal S_{\mathcal L}^0\equiv\mathcal S^0_{\mathcal L_{\mathbf R}}$. Moreover, if $\u$ minimizes $\mathcal G$ (resp. $\mathcal G^I$) over $W^{1,p}(\om,\R^3)$ and $\mathbf R\in\mathcal S_{\mathcal L}^0$ realizes the maximum in the definition of $\mathcal G(\u)$ (resp. $\mathcal G^I(\u)$), then $\u$ minimizes $\mathcal G_{\mathbf R}$ (resp. $\mathcal G^I_{\mathbf R}$) over $W^{1,p}(\om,\R^3)$  and 
\[
\min_{W^{1,p}(\om\R^3)}\mathcal G_{\mathbf R}=\min_{W^{1,p}(\om\R^3)}\mathcal E_{\mathbf R}\qquad(\mbox{resp.}\;\min_{W^{1,p}(\om\R^3)}\mathcal G^I_{\mathbf R}=\min_{W^{1,p}(\om\R^3)}\mathcal E^I_{\mathbf R}).
\]
 \end{theorem}


\begin{remark}\label{29} \rm
Given any external forces $\mathbf f,\mathbf g$ satisfying \eqref{globalequi}-\eqref{comp}, the
 above theorem yields the existence of new external forces satisfying \eqref{globalequi}-\eqref{comp}, having the same rotation kernel as $\mathbf f,\mathbf g$,   for which there is no gap with linear elasticity.
\end{remark}
 
 We include in this section  the straightforward proof of Theorem \ref{rotheorem}.


  \begin{proofad4}
  We give the proof for the compressible case, the arguments for the incompressible case being the very same.
  
  By Remark \ref{29}, $\mathcal L_{\mathbf R}$ satisfies   \eqref{globalequi} and \eqref{comp}. If $\mathbf R\in\mathcal S_{\mathcal L}^0$ we notice that, since $\mathcal S_{\mathcal L}^0$ and $\mathcal S_{\mathcal L_{\mathbf R}}^0$ are subgroups of $SO(3)$ as shown in Remark \ref{subgroup},  inequality \eqref{eqi} is an equality as soon as $\mathbf S\in\mathcal S^0_{\mathcal L}$, so that  $\mathcal S_{\mathcal L_{\mathbf R}}^0\supseteq\mathcal S^0_{\mathcal L}$. Still assuming $\mathbf R\in\mathcal S_{\mathcal L}^0$, we may also prove the opposite inclusion: indeed, by the inclusion $\mathcal S_{\mathcal L_{\mathbf R}}^0\supseteq\mathcal S^0_{\mathcal L}$ we deduce that $\mathbf R\in\mathcal S^0_{\mathcal L_{\mathbf R}}$, so that again Remark \ref{subgroup} implies $\mathbf R^T\in\mathcal S^0_{\mathcal L_{\mathbf R}}$ and $\mathbf R^T\mathbf S\in\mathcal S^0_{\mathcal L_{\mathbf R}}$ for every $\mathbf S\in\mathcal S^0_{\mathcal L_{\mathbf R}}$, thus
  \[
  \mathcal L((\mathbf S-\mathbf I)\x)=\mathcal L_{\mathbf R}(\mathbf R^T(\mathbf S-\mathbf I)\x)=\mathcal L_{\mathbf R}((\mathbf R^T\mathbf S-\mathbf I)\x)-\mathcal L_{\mathbf R}((\mathbf R^T-\mathbf I)\x)=0
  \]
  for every $\mathbf S\in\mathcal S^0_{\mathcal L_{\mathbf R}}$, proving that $\mathcal S_{\mathcal L_{\mathbf R}}^0\subseteq\mathcal S^0_{\mathcal L}$.
%
%
%

Let now $\mathbf R\in\mathcal S_{\mathcal L}^0$ realize the maximum in the definition of $\mathcal G(\u)$, where $\u$ minimizes $\mathcal G$ over $W^{1,p}(\om,\R^3)$.
We conclude by checking that $\u$ is also a minimizer of $\mathcal G_{\mathbf R}$ over $W^{1,p}(\om,\R^3)$ and that the identity matrix realizes the maximum in the definition of $\mathcal G_{\mathbf R}(\u)$.
   We proceed by contradiction,
    supposing that there are $\tilde \u\in W^{1,p}(\om,\mathbb R^3)$ and  $\tilde{\mathbf R}\in\mathcal S_{\mathcal L_{\mathbf R}}^0$ such that
\[\begin{aligned}\min_{W^{1,p}(\om,\mathbb R^3)}\mathcal G_{\mathbf R}&=
\int_\om\mathcal Q(\x,\mathbb E(\tilde\u))\,d\x-\mathcal L_{\mathbf R}(\tilde{\mathbf R}\tilde \u)
<
\int_\om\mathcal Q(\x,\mathbb E(\u))\,d\x-\mathcal L_{\mathbf R}(\u).
\end{aligned}\]
 Then, having shown that $\mathbf R\tilde{\mathbf R}\in\mathcal S_{\mathcal L}^0$, we deduce
\[\begin{aligned}
 \mathcal G(\tilde \u)&=\int_\om\mathcal Q(\x,\mathbb E(\tilde \u))\,d\x-\max_{\mathbf S\in\mathcal S_{\mathcal L}^0}\mathcal L(\mathbf S\tilde\u)\le\int_\om\mathcal Q(\x,\mathbb E(\tilde \u))\,d\x-\mathcal L(\mathbf R\tilde{\mathbf R}\tilde\u)\\&=\int_\om\mathcal Q(\x,\mathbb E(\tilde \u))\,d\x-\mathcal L_{\mathbf R}(\tilde{\mathbf R}\tilde\u)<\int_\om\mathcal Q(\x,\mathbb E(\u))\,d\x-\mathcal L_{\mathbf R}(\u)
 \\&=
 \int_\om\mathcal Q(\x,\mathbb E(\u))\,d\x-\mathcal L(\mathbf R\u)
 =\mathcal G(\u),
\end{aligned}\] 
which is a contradiction with the minimality of $\u$ for $\mathcal G$.
\end{proofad4}

\LLL\PPP
\KKK
\begin{remark}\label{maormora} \rm We close this section by mentioning a difference between our approach to Theorem \ref{mainth1comp} and the one   in \cite{MM},  where the authors introduce the set $ \mathcal R:=\argmax_{\mathbf R\in SO(3)} \mathcal L(\mathbf R \x)$. Under the usual assumptions on $\mathcal W$ (with $p=2$)  and assuming only \eqref{globalequi}, it follows from \cite[Theorem 5.3]{MM} that for every $\mathbf U\in\mathcal R$ the infimum of  
\beeq\lab{energiafinta}
\mathcal J_{h,\mathbf U}(\mathbf y): =h^{-2}\int_\om\mathcal W(\x,\nabla\mathbf y)\,d\x
-h^{-1}\mathcal L(\mathbf y-{\mathbf U} \x)
\eneq
among all $\u\in H^1(\om,\mathbb R^3)$ converges as $h\to 0$ to the minimum over $\u\in H^1(\om,\mathbb R^3)$ of
\begin{equation*}
\mathcal J_{\mathbf U}(\u)=\displaystyle \int_\om \mathcal Q(\x, \mathbb E(\u))\,d\x- \max_{\mathbf R\in \mathcal R} \mathcal L(\mathbf U\mathbf R\u).
\end{equation*}
It is readily seen that if $\mathcal L$ satisfies \eqref{globalequi}  and $ \mathbf I\in \mathcal R$ then $\mathcal L$ satisfies  \eqref{comp}, $ \mathcal R\equiv 
{\mathcal  S}^0_{\mathcal L}$ and $\mathcal J_{\mathbf I}(\u)=\mathcal G(\u)$ for every $\u\in H^1(\om,\mathbb R^3)$. 
Moreover
\begin{equation*}
\mathcal J_{h,\mathbf U}(\mathbf y)= \mathcal G_h(\mathbf y)+ h^{-1}\mathcal L({\mathbf U} \x-\x)\ge \mathcal G_h(\mathbf y),
\end{equation*}
and equality holds if and only if $\mathbf I\in \mathcal R$ so that in this case $\inf \mathcal G_h\to \min \mathcal G=\min\mathcal J_{\mathbf I}$. Therefore the results of \cite{MM} imply Theorem \eqref{mainth1comp} when $p=2$.
 On the other hand, if $\mathbf I\not\in \mathcal R$ (so that \eqref{comp} does not hold) then $\inf \mathcal G_h\to -\infty$ as $h\to 0$ (see Remark 2.5) and therefore hypothesis \eqref{comp} cannot be dropped in Theorem \eqref{mainth1comp}.

In addition, we claim that  whenever the sole condition \eqref{globalequi} holds along with the usual assumptions on $\mathcal W$ with $p\in (1,2]$, Theorem \eqref{mainth1comp} implies that for every $\mathbf U\in \mathcal R$ there holds $\inf \mathcal J_{h,\mathbf U}\to \min \mathcal J$ as $h\to 0$, where
\begin{equation*}
\mathcal J(\u)=\displaystyle \int_\om \mathcal Q(\x, \mathbb E(\u))\,d\x- \max_{\mathbf R\in \mathcal R} \mathcal L(\mathbf R\u).
\end{equation*}
Indeed, we first notice that if $\mathcal L$ satisfies \eqref{globalequi}, then ${\mathbf U}\in\mathcal R$ implies that $\mathcal L_{{\mathbf U}}$ satisfies both \eqref{globalequi} and \eqref{comp} and that 
\beeq
\lab{RsegnatoR}\mathcal R=\{{\mathbf U}\mathbf R:\mathbf R\in\mathcal S_{\mathcal L_{{\mathbf U}}}^0\}.
\eneq
In fact, given ${\mathbf U}\in\mathcal R$, if $\mathbf S\in \mathcal S_{\mathcal L_{{\mathbf U}}}^0$ it is immediately seen that ${\mathbf U}\mathbf S\in\mathcal R$, and given any $\mathbf S_*\in\mathcal R$ we may write $\mathbf S_*=\mathbf U\,{\mathbf U}^T\mathbf S_*$ and it is immediately seen that ${\mathbf U}^T\mathbf S_*\in\mathcal S^0_{\mathcal L_{{\mathbf U}}}$, thus proving \eqref{RsegnatoR}. Moreover we notice that
\begin{equation*}\lab{energiaequiv}
\mathcal J_{h,\mathbf U}(\mathbf y)= \mathcal G_{h,\mathbf U}(\mathbf U^T\mathbf y)=h^{-2}\int_\om\mathcal W(\x,\nabla(\mathbf U^T\mathbf y))\,d\x
-h^{-1}\mathcal L_{\mathbf U}(\mathbf U^T\mathbf y- \x),
\end{equation*}
thus $\inf\mathcal J_{h,\mathbf U}=\inf \mathcal G_{h,\mathbf U}$, and by recalling that $\mathcal L_{{\mathbf U}}$ satisfies both \eqref{globalequi} and \eqref{comp}, it follows from Theorem \ref{mainth1comp} that $\inf\mathcal J_{h,\mathbf U}$  converges, as $h\to 0$, to the minimum of
\begin{equation*}
{\mathcal G}_{\mathbf U}(\u):=\displaystyle \int_\om \mathcal Q(\x, \mathbb E(\u))\,d\x-\max_{\mathbf R\in \mathcal S_{\mathcal L_{{\mathbf U}}}^0} \mathcal L_{{\mathbf U}}(\mathbf R\u)
\end{equation*}
among all $\u\in H^1(\om,\mathbb R^3)$. By recalling \eqref{elleerre} and by exploiting \eqref{RsegnatoR} we also get 
\begin{equation*}
\max_{\mathbf R\in \mathcal S_{\mathcal L_{{\mathbf U}}}^0} \mathcal L_{{\mathbf U}}(\mathbf R\u)=\max_{\mathbf R\in \mathcal S_{\mathcal L_{{\mathbf U}}}^0} \mathcal L({\mathbf U}\mathbf R\u)=\max_{\mathbf R\in \mathcal R} \mathcal L(\mathbf R\u),
\end{equation*}
that is, ${\mathcal G}_{\mathbf U}(\u)\equiv \mathcal J(\u)$ and $\inf \mathcal J_{h,\mathbf U}\to \min \mathcal J$ thus proving the claim (i.e., $\min \mathcal J=\min \mathcal J_{\mathbf U}$ for every $\mathbf U\in\mathcal R$). 

In any case, we observe that
if $\om$ is the reference configuration of the elastic body, the second term on the right hand side of \eqref{energiafinta} represents the work expended by the given external forces $\mathbf f ,\mathbf g$ if and only  $\mathbf I\in\mathcal R$
or equivalently if and only if \eqref{comp} is satisfied by $\mathbf f,\mathbf g$. 
\end{remark}

\KKK

 \section{Preliminary results}\label{prel}
 \subsection*{Some properties of $\mathcal W$} 
The frame indifference assumption \eqref{framind} implies that there exists a function $\mathcal V$ such that  for a.e. $\x\in \om$
\begin{equation}\label{vi}
\W(\x,\mathbf F)=\V(\x,\textstyle{\frac{1}{2}}( \mathbf F^T \mathbf F - \mathbf I))\,
\quad
\ \forall\, \mathbf F\in \mathbb R^{3\times 3}.
\end{equation}
 By  \eqref{reg}, for a.e. $\x\in\Omega$, we have
 $
 \mathcal W(\x,\mathbf R)=D\mathcal W(\x,\mathbf R)=0$ for any $\mathbf R\in SO(3).
 $
 By \eqref{vi}, for a.e. $\x\in\om$,
given $\mathbf B\in\mathbb R^{3\times 3}$ and $h> 0$ we have
 $
\mathcal W(\x,\Id+h\mathbf B)=\V(\x,h\,{\rm sym}\mathbf B+\tfrac12h^{2}\mathbf B^{T}\mathbf B)
$
 and \eqref{reg} again implies
\[\displaystyle 
\lim_{h\to 0} h^{-2}\mathcal W(\x,\mathbf I+h\mathbf B)=
\frac{1}{2} \,{\rm sym} \mathbf  B\, D^2\V (\x, \mathbf 0) \ {\rm sym}\mathbf  B=\frac12\, \mathbf B^T D^2\mathcal W(\x,\mathbf I)\,\mathbf B,
\qquad\forall \,\mathbf B\in\mathbb R^{3\times3}.
\]
By the latter and by \eqref{coerc}, for a.e. $\x\in \om$, the following holds for every $\mathbf B\in\mathbb R^{3\times3}$ with $\det\mathbf B>0$:
\begin{equation*}\label{ellipticity}\begin{aligned}
\frac12\, \mathbf B^T D^2\mathcal W(\x,\mathbf I)\,\mathbf B&=\lim_{h\to 0} h^{-2}\mathcal W(\x,\mathbf I+h\mathbf B)
\ge \limsup_{h\to 0}Ch^{-2}\,d^2(\mathbf I+h\mathbf B,SO(3))\\&=\limsup_{h\to 0} C h^{-2}\left|\sqrt{(\mathbf I+h\mathbf B)^T(\mathbf I+h\mathbf B)}-\mathbf I\right|^2
=C|\mathrm{sym}\mathbf B|^2.\end{aligned}
\end{equation*}
%
 Moreover, as noticed also in \cite{MP2}, by expressing the remainder of Taylor's expansion in terms of the $\x$-independent modulus of continuity $\omega$ of $D^2\W(\x,\cdot)$ on the set $\mathcal U$  \KKK from \eqref{reg}, we have 
\begin{equation}\lab{regW}
\left|\mathcal W(\x, \mathbf I+h\mathbf B)- \frac{h^2}{2} \,{\rm sym} \mathbf  B\, D^2\mathcal W (\x, \mathbf I) \ {\rm sym}\mathbf  B\right|\le h^2\omega(h|\mathbf B|)|\mathbf B|^2
\end{equation}
for any small enough $h$ (such that $h\mathbf B\in\mathcal U$). Similarly, $\mathcal V(\x,\cdot)$ is $C^2$ in a neighbor of the origin in $\mathbb R^{3\times 3}$, with an $\x$-independent modulus of continuity $\eta:\mathbb R_+\to \mathbb R$, which is increasing and such that $\lim_{t\to0^+}\eta(t)=0$, and  we have
\begin{equation}\lab{regV}
\left|\mathcal V(\x, h\mathbf B)- \frac{h^2}{2} \,{\rm sym} \mathbf  B\, D^2\V (\x, \mathbf 0) \ {\rm sym}\mathbf  B\right|\le h^2\eta(h|\mathbf B|)|\mathbf B|^2
\end{equation}
for any small enough $h$.

 \subsection*{Some functional inequalities}
  Let $p\in(1,2]$. Since $\Omega$ is a bounded open connected Lipschitz set,
 by Sobolev embedding, Sobolev trace embedding and by the Poincar\'e inequality  for any $\v\in W^{1,p}(\Omega,\mathbb R^3)$ there exists $\bar{\mathbf c}, \bar{\mathbf d}\in\mathbb R^3$ such that
 \begin{equation}\label{fried}
 \|\v-\bar{\mathbf c}\|_{L^{\frac{3p}{3-p}}(\Omega,\mathbb R^3)}+\|\v-\bar{\mathbf d}\|_{L^{\frac{2p}{3-p}}(\partial\Omega,\mathbb R^3)}\le K\|\nabla \v\|_{L^p(\Omega,\mathbb R^{3\times3})},
 \end{equation} 
 where $K$ is a constant only depending on $\Omega,p$.  Moreover, the second Korn inequality (see for instance \cite{N}), 
   combined with  Sobolev and trace inequalities, provides the existence of a further constant $C_K=C_K(\om,p)$ such that for all $\v\in W^{1,p}(\om,\R^3)$
\begin{equation}\label{kornvera}
\|\vv-\mathbb P\vv\|_{L^{\frac{3p}{3-p}}(\Omega,\R^3)}\,+\,
\|\vv-\mathbb P\vv\|_{L^{\frac{2p}{3-p}}(\partial\Omega,\R^3)}
\,\le\ C_{K}\ \|\mathbb E(\vv)\|_{L^p(\Omega,\mathbb R^{3 \times 3})},
\end{equation}
where $\mathbb P$ denotes the projection operator on infinitesimal rigid displacements, i.e., on the set of displacement fields $\v$ such that $\mathbb E(\v)=0$.

A useful consequence of \eqref{fried}, if  \eqref{globalequi} holds true, is the following estimate. Since for any $\v\in W^{1,p}(\Omega,\mathbb R^3)$  and for every  $\mathbf c, \mathbf d\in \mathbb R^3$  \begin{equation*}
|\mathcal L(\v)|\le \|\mathbf f\|_{L^{\frac{3p}{4p-3}}(\Omega,\mathbb R^3)}\|\v-\mathbf c\|_{L^{\frac{3p}{3-p}}(\Omega,\mathbb R^3)} + \|\mathbf g\|_{L^{\frac{2p}{3p-3}}(\partial\Omega,\mathbb R^3)}\|\v-\mathbf d\|_{L^{\frac{2p}{3-p}}(\partial\Omega,\mathbb R^3)} \\
\end{equation*}
then \eqref{fried} implies
\begin{equation*}\label{bella}
|\mathcal L(\v)|\le C_{\mathcal L}\|\nabla\v\|_{L^p(\Omega,\R^{3\times3})}
\end{equation*}
where $C_{\mathcal L}:=K\left(\|\mathbf f\|_{L^{\frac{3p}{4p-3}}(\Omega,\mathbb R^3)}+ \|\mathbf g\|_{L^{\frac{2p}{3p-3}}(\partial\Omega,\mathbb R^3)}\right)$ and $K$ is the constant in \eqref{fried}. By Young inequality we then obtain
\begin{equation}\label{elle}
|\mathcal L(\v)|\le \frac{p-1}{p}(C_{\mathcal L}\eps^{-1})^{\frac{p}{p-1}}+p^{-1}\eps^p\|\nabla\v\|^{p}_{L^p(\Omega,\R^{3\times3})}
\end{equation}
for every $\eps > 0$.


\section{Convergence of minimizers: proof of Theorem \ref{mainth1comp} and Theorem \ref{mainth1}}
\label{proofsection}

\subsection{The incompressible case} We give the proof of our convergence result regarding the incompressible case, which is the more difficult.   We will briefly show how to adapt the arguments to the compressible case later on. The proof follows the standard line of a $\Gamma$-convergence argument: we prove compactness, a lower bound and an upper bound. \KKK

\begin{lemma}\label{compactness}{\bf (Compactness)}.
Assume  \eqref{framind}, \eqref{Z1}, \eqref{reg}, \eqref{coerc}, 
   \eqref{globalequi} and \eqref{comp}.
 Let $(h_{j})_{j\subset\mathbb N}\subset(0,1)$ be a vanishing a sequence, let $M>0$   
 and let $(\mathbf y_{{j}})_{j\in\mathbb N}\subset W^{1,p}(\Omega,\R^3)$ be a sequence  such that \begin{equation}\lab{boundG}\mathcal G^I_{h_{j}}(\mathbf y_{j})\le M\qquad\forall \ j\in\mathbb N.\end{equation} 
 Let   $\mathbf R_j\in\mathcal A_p(\mathbf y_j)$ \KKK and $\u_j(\x):=h_j^{-1}(\mathbf R_j^T\mathbf y_j(\x)- \x)$.
  Then, the sequence $(\nabla \u_j)_{j\in\mathbb N}$ is bounded in $L^p(\om,\R^{3\times3})$ and any of its weak $L^p(\om,\R^{3\times3})$ limit points is of the form $\nabla\u_*$ for some $\u_*\in H^1(\om,\R^3)$. Moreover, any limit point of the sequence $(\mathbf R_j)_{j\in\mathbb N}\subset SO(3)$ belongs to $\mathcal S^0_{\mathcal L}$. \KKK
\end{lemma}
\begin{proof} 
 By \eqref{boundG} we obtain $\det\nabla\mathbf y_j=1$ for any $j\in\mathbb N$. \KKK
By \eqref{boundG}, \eqref{globalequi}, \eqref{comp} and \eqref{elle} we get for every $\eps > 0$
\begin{equation}\lab{bdw}\begin{aligned}
\displaystyle h_j^{-2}\int_\om\mathcal W^I(\x,\nabla\mathbf y_j)\,d\x&=h_j^{-2}\int_\om\mathcal W^I(\x,\nabla\mathbf y_j)\,d\x\le M+
h_j^{-1}\mathcal L(\mathbf y_j-\x)\\
&= M+h_j^{-1}\mathcal L(\mathbf y_j-\mathbf R_j\x)+\displaystyle h_j^{-1}\mathcal L((\mathbf R_j-\mathbf I)\x)\\
&\displaystyle \le M+h_j^{-1}\mathcal L(\mathbf y_j-\mathbf R_j\x)\\&\le  \frac{p-1}{p}(C_{\mathcal L}\eps^{-1})^{\frac{p}{p-1}}+p^{-1}\eps^p\|\mathbf R_j\nabla\u_j\|^{p}_{L^p(\Omega,\R^{3\times3})}.
\end{aligned}
\end{equation}
Moreover, by recalling \eqref{coerc}, \eqref{muller} and \eqref{propgp}, \KKK there exists a constant $C$ (only depending on $p$ and $\Omega$) such that for every $\eps>0$
\begin{equation}\label{odds}\begin{aligned}
&h_j^{-2}\int_\Omega\mathcal W^I(\x,\nabla \mathbf y_j)\,d\x=h_j^{-2}\int_\om\mathcal W^I(\x,\nabla\mathbf y_j)\,d\x\ge Ch_j^{-2}\int_\Omega g_p(|\nabla\mathbf y_j-\mathbf R_j|)\,d\x\\&\qquad=
Ch_j^{-2}\int_\Omega g_p(h_j |\mathbf R_j\nabla \mathbf y_j|)\,d\x \ge C \|\mathbf R_j\nabla \mathbf u_j\|^p_{L^p(\Omega,\R^{3\times3})}-\frac{2-p}{p}\,C\,|\Omega|,
\end{aligned}\end{equation}
which in combination with \eqref{bdw} entails, by taking small enough $\eps$,
\begin{equation}\label{bdu}
\|\mathbf R_j\nabla\u_j\|^{p}_{L^p(\Omega,\R^{3\times3})}=\|\nabla\u_j\|^{p}_{L^p(\Omega,\R^{3\times3})}\le Q
\end{equation}
for some suitable constant $Q$ depending only on $C_{\mathcal L}, p,\Omega$ (and not on $j$).
 On the other hand 
\begin{equation*}
\displaystyle \mathcal G_h^{I}(\mathbf y_j):=h_j^{-2}\int_\om\mathcal W^I(\x,\nabla\mathbf y_j)\,d\x
-\mathcal L(\mathbf R_j\u_j)-h_j^{-1}\mathcal L((\mathbf R_j-\mathbf I)\x)\le M
\end{equation*}
entails
\begin{equation}\lab{esse0}
0\le -h_j^{-1}\mathcal L((\mathbf R_j-\mathbf I)\x)\le M+\mathcal L(\mathbf R_j\u_j)
\end{equation}
and by \eqref{elle}, \eqref{bdu} we get $\mathcal L((\mathbf R_j-\mathbf I)\x)\to 0$ as $j\to+\infty$. Therefore, if  $\mathbf R_j\to \mathbf R_*$ along a suitable subsequence, we have $\mathcal L((\mathbf R_*-\mathbf I)\x)= 0$ that is $\mathbf R_*\in \mathcal S^0_{\mathcal L}$. 

By \eqref{bdu}, the sequence $(\nabla \u_j)_{j\in\mathbb N}$ is bounded in $L^p(\om,\R^{3\times3})$. As a consequence of the Poincar\'e inequality, any of its weak $L^p(\Omega,\R^{3\times3})$ limit points is of the form $\nabla\u_*$ for some  $\u_*\in W^{1,p}(\Omega,\R^3)$.  Assuming that $\nabla\u_*$ is the weak $L^p(\Omega,\R^{3\times3})$ limit point along a not relabeled  subsequences, 
we are only left to prove that $\u_*\in H^1(\om;\mathbb R^3)$. To this aim we let
\begin{equation}\label{BJ}
B_j:=\{\x\in \om: \sqrt h_j|\nabla\u_j|\le 1\}
\end{equation}
so that 
\begin{equation*}
\int_{B_j}|\nabla\u_j|^2\,d\x\le h_j^{-2}\int_\om g_p(h_j|\nabla\u_j|)\,d\x=h_j^{-2}\int_\Omega g_p(|\nabla \mathbf y_j-\mathbf R_j|)\,d\x,
\end{equation*}
hence by \eqref{bdw}, \eqref{odds} and \eqref{bdu} we get uniform boundedness in $L^2(\om,\mathbb R^{3\times 3})$ for the sequence $(\mathbf 1_{B_j}\nabla \mathbf u_j)_{j\in\mathbb N}$
thus up to subsequences $\mathbf 1_{B_j}\nabla\u_j\wconv \w$ weakly in $L^2(\om,\mathbb R^{3\times 3})$ as $j\to+\infty$.
On the other hand for every $q\in (1,p)$ we have
\begin{equation*}
\displaystyle\int_{B_j^c}|\nabla\u_j|^q\,d\x\le \left(\int_{B_j^c}|\nabla\u_j|^p\,d\x\right)^{q/p}|B_j^c|^{(p-q)/p}
\end{equation*}
where the right hand side vanishes as $j\to+\infty$ 
since $|B_j^c|\to 0$ by Chebyshev inequality. 
By taking into account that
\begin{equation*}
\nabla \u_j={\mathbf 1}_{B_j^c}\nabla\u_j+{\mathbf 1}_{B_j}\nabla\u_j
\end{equation*}
and  that ${\mathbf 1}_{B_j}\nabla\u_j\wconv \w$ weakly in $L^2(\om,\R^{3\times3})$ we get $\nabla\u_j\wconv \w$
weakly in $L^q(\om,\R^{3\times3})$ and recalling that $\nabla\u_j\wconv \nabla\u_*$ weakly in $L^p(\om,\mathbb R^{3\times3})$ we get $\w=\nabla\u_*\in L^2(\om,\R^{3\times3})$ thus proving that $\u_*\in H^1(\om,\R^3)$.  
\end{proof}

%
%


 \begin{lemma}\label{lowerbd}{\bf (Lower bound)}. Assume \eqref{globalequi}, \eqref{comp},  \eqref{framind}, \eqref{Z1}, \eqref{reg},    
\eqref{coerc}.
 Let $(\mathbf y_j)_{j\in\mathbb N}\subset W^{1,p}(\om,\R^3)$ be a sequence. For any $j\in\mathbb N$, let  $\mathbf R_j\in\mathcal A_p(\mathbf y_j)$ \KKK and  $\u_j(\x):=h_j^{-1}(\mathbf R_j^T\mathbf y_j(\x)-\x)$. Suppose that there exists $\u\in W^{1,p}(\om,\R^3)$ such that \KKK
  $\nabla\u_j\wconv \nabla\u$ weakly in $L^p(\om, \mathbb R^3)$. Then
\begin{equation*}\lab{liminf}
\displaystyle\liminf_{j\to +\infty}\mathcal G_{h_j}^I(\mathbf y_j)\ge \mathcal G^I(\u).
\end{equation*}
\end{lemma}
\begin{proof} We may assume wlog that ${\mathcal G^I_{h_j}}(\mathbf y_j)\le M$ for any $j\in\mathbb N$ hence $\u\in H^1(\om;\mathbb R^3)$ by Lemma   \eqref{compactness}  and
\begin{equation*}\begin{aligned}
1&=\det\nabla\mathbf y_j=\det(\mathbf R_j(\mathbf I+h_{j}\nabla\u_{j}))=\det(\mathbf I+h_{j}\nabla\u_{j})=\\
&= 1+h_{j}\dv\u_j-\frac{1}{2}h_{j}^{2}(\mathrm{Tr}(\nabla\u_{j})^{2}-(\Tr \nabla\u_{j})^{2})+h_{j}^{3}\det \nabla\u_j
\end{aligned}
\end{equation*}
a.e. in $\om$, that is, 
\begin{equation*}\lab{TrB}
\dv \u_{j}=
\frac{1}{2}h_{j}(\mathrm{Tr}(\nabla\u_{j})^{2}-(\Tr \nabla\u_{j})^{2})-h_{j}^{2}\det \nabla\u_j.
\end{equation*}
By taking into account that $\nabla\u_j$ are uniformly bounded in $L^p$ we get $h_j^{\alpha}|\nabla\u_j|\to 0$ a.e. in $\om$ for every $\alpha > 0$ hence
 $\dv \u_{j}=\frac{1}{2}h_{j}(\mathrm{Tr}(\nabla\u_{j})^{2}-(\Tr \nabla\u_{j})^{2})-h_{j}^{2}\det \nabla\u_j\to 0$ a.e. in $\om$. 
Since the weak convergence of $\nabla\u_j$ implies $\dv \u_{j}\wconv \dv \u$ weakly in $L^p(\Omega)$ we get $\dv\u=0$ a.e. in $\om$.
By  setting $$\textstyle\mathbf D_{j}:=\mathbb E(\u_{j})+\frac{1}{2}h_{j}\nabla\u_{j}^{T}\nabla\u_{j},$$ 
by \eqref{regV}, \eqref{vi} \eqref{globalequi} and \eqref{comp}, and by recalling that $B_j$ is defined in \eqref{BJ}, we get for large enough $j$
\begin{equation}\label{lowerbd}\begin{aligned}
\mathcal G_{h_j}^I(\v_j)&\ge\frac1{h_j^2}\int_{B_j}\mathcal V(\x,h_j \mathbf D_j)\,d\x-\mathcal L(\mathbf R_j\u_j)-h_j^{-1}\mathcal L((\mathbf R_j-\mathbf I)\x)\\
&\displaystyle\ge \int_{B_j}\frac12 \mathbf D_j^T\,D^2\mathcal V(\x,\mathbf 0)\,\mathbf D_j\,d\x-\int_{B_j} \eta(h_j\mathbf D_j)|\mathbf D_j|^2\,d\x-\mathcal L(\mathbf R_j\u_j)
\\&\ge \frac12\int_\om ({\mathbf 1}_{B_j}\mathbf D_j)^T\,D^2\mathcal W(\x,\mathbf I)\,(\mathbf 1_{B_j}\mathbf D_j)\,d\x- \eta(\sqrt h_j)\int_\om|{\mathbf 1}_{B_j}\mathbf D_j|^2\,d\x-\mathcal L(\mathbf R_j\u_j),
\end{aligned}
\end{equation}
since on $B_j$ we have $h_j|\mathbf D_j|\le \sqrt{h_j}\left(\sqrt{h_j}|\nabla\v_j|+\tfrac12h_j^{3/2}|\nabla \v_j^T||\nabla \v_j|\KKK\right)\le 2\sqrt{h_j}$ for large enough $j$ (so that indeed \eqref{regV} can be applied\KKK) and since $\eta$ is increasing. 
Since $h_{j}\nabla\u_{j}^T\nabla\u_j\to 0$ a.e. in $\om$ and $|B_j^c|\to 0$ as $j\to+\infty$, and since $|{\mathbf 1}_{B_j}h_{j}\nabla\u_{j}^T\nabla\u_j| \le 1$, we get ${\mathbf 1}_{B_j}h_{j}\nabla\u_{j}^T\nabla\u_j\wconv 0$ weakly in $L^2(\Omega,\mathbb R^{3\times3})$. By  taking into account that ${\mathbf 1}_{B_j}\nabla\u_j\wconv \nabla\u$ weakly in $L^2(\Omega,\mathbb R^{3\times3})$, we then obtain
$\textstyle{\mathbf 1}_{B_j}\mathbf D_j\wconv \mathbb E(\u)$ weakly in $L^2(\Omega,\mathbb R^{3\times 3})$. 
Let now $\mathbf c_j\in \mathbb R^3$ such that $\u_j-\mathbf c_j\wconv \u$ weakly in $W^{1,p}(\Omega,\mathbb R^{3\times3})$. {By taking into account that, up to subsequences, Lemma \ref{compactness} entails  $\mathbf R_j\to \mathbf R\in \mathcal S^0_{\mathcal L},$ we get}
\begin{equation*}
\lim_{j\to +\infty}\mathcal L(- \mathbf R_j \u_j)=\lim_{j\to +\infty}\mathcal L(-\mathbf R_j (\u_j-\mathbf c_j))= -\mathcal L( \mathbf R \u).
\end{equation*}
Hence, by \eqref{reg}, \eqref{lowerbd} and by the weak $L^2(\Omega,\mathbb R^{3\times3})$ lower semicontinuity of the map $\mathbf F\mapsto\int_\Omega \mathbf F^T\, D^2\mathcal W(\x,\mathbf I)\,\mathbf F\,d\x$, we conclude   
\begin{equation*}\label{pi}
\displaystyle\liminf_{j\to +\infty}\mathcal G_{h_j}^I(\v_j)\ge \frac12\int_\om \mathbb E(\u)\,D^2\mathcal W(\x,\mathbf I)\mathbb E(\u)\,d\x-\mathcal L(\mathbf R\u)\ge \mathcal G^I(\u)
\end{equation*}
which ends the proof.
\end{proof}

We next provide the construction for the recovery sequence, taking advantage of the following approximation result from  \cite{MP2}.

\begin{lemma}[{\cite[Lemma 6.2]{MP2}}]\label{Kato} Suppose that $\partial\Omega$ has a finite number of connected components.
Let $(h_j)_{j\in\mathbb N}\subset(0,1)\KKK$ be a vanishing sequence.
Let $\u\in H^1_{\dv}(\Omega,\R^3)$. There exists a sequence $(\u_j)_{j\in\mathbb N}\subset W^{2,\infty}(\om,\mathbb R^3)$ such that
\begin{itemize}
\item[i)]
 $\det(\mathbf I+h_j\nabla \u_j)=1$ for any $j\in\mathbb N$,
 \item[ii)]  $h_j\|\nabla\u_j\|_{L^\infty(\Omega)}\to 0$ as $j\to+\infty$,
  \item[iii)]  $\u_j\to\u$ strongly in $H^1(\Omega,\mathbb R^3)$ as $j\to+\infty$.
 \end{itemize}
\end{lemma}

\begin{lemma}\lab{upbd}{\bf (Upper bound)}. Suppose that $\partial\Omega$ has a finite number of connected components. Assume  \eqref{framind}, \eqref{Z1}, \eqref{reg},    
\eqref{coerc}. Let $(h_j)_{j\in\mathbb N}\subset(0,1)$ be a vanishing sequence. For every $\u\in W^{1,p}(\Omega,\mathbb R^3)$ there exists a sequence $(\mathbf u_j)_{j\in\mathbb N}\subset W^{1,p}(\Omega,\mathbb R^3)$ such that 
$\u_j\wconv \u \hbox { \rm  weakly in}\  W^{1,p}(\Omega,\mathbb R^3)$ as $j\to+\infty$ and $\mathbf R_*\in \mathcal S^0_{\mathcal L}$ such that by setting $\mathbf y_j:=\mathbf R_*(\x+h_j\u_j)$ we have
\[ \limsup_{j\to +\infty} {\mathcal G}_{h_j}^I(\mathbf y_j) \le {\mathcal G}^I(\u).\]
\end{lemma}
\begin{proof}
It is enough to prove the result in case $\u\in H^1_{\mathrm{div}}(\Omega,\mathbb R^3)$. We take the sequence $(\u_j)_{j\in\mathbb N}$ from Lemma \ref{Kato} so that $\u_j\to\u$ strongly in $H^1(\om,\R^3)$ as $j\to+\infty$,
and we take
$$\mathbf R_*\in \argmin\left\{ \int_\Omega \mathcal Q^I(\x, E(\u))\,d\x-\mathcal L(\mathbf R\u): \mathbf R\in \mathcal S^0_{\mathcal L}\right\}.$$
We set $\mathbf y_j:=\mathbf R_*(\x+h_j\u_j)$ and  $\mathcal F(\u_j):=\frac12\int_\om\nabla\u_j^T D^2\mathcal W(\x,\mathbf I) \nabla\u_j\,d\x-\mathcal L(\mathbf R_*\,\u_j)$ for any $j\in\mathbb N$.
Property ii) of Lemma \ref{Kato} yields   $\mathbf I+h_j\nabla\u_j\in\mathcal U$ for a.e. $\x$ in $\Omega$ if $j$ is large enough, where $\mathcal U$ is the neighbor of $SO(3)$ that appears in \eqref{reg}.    
In particular, $D^2\mathcal W(\x,\cdot)\in C^2(\mathcal U)$ for a.e. $\x\in\om$ 
 and we make use of \eqref{regW} together with
  $\det(\mathbf I+h_j\nabla\u_j)=1$ and $\mathbf R_*\in \mathcal S^0_{\mathcal L}$ to obtain 
  \begin{equation*}\label{nee}\begin{aligned}
\displaystyle\limsup_{j\to+\infty} \mathcal |\mathcal G^I_{h_j}(\mathbf y_j)-\mathcal F(\u_j)|
&\displaystyle\le \limsup_{j\to+\infty}\int_{\Omega} \left|\frac{1}{h_j^2}\mathcal W^I(\x,\mathbf I+h_j\nabla\u_j)-\frac12\,\nabla\u_j^T D^2\mathcal W(\x,\mathbf I) \nabla\u_j\right|\,d\x\\
&=\displaystyle\limsup_{j\to+\infty}\int_{\Omega} \left|\frac{1}{h_j^2}\mathcal W(\x,\mathbf I+h_j\nabla\u_j)-\frac12\,\nabla\u_j^T D^2\mathcal W(\x,\mathbf I) \nabla\u_j\right|\,d\x\\
&\displaystyle\le\limsup_{j\to+\infty}\int_{\Omega}\omega(h_j|\nabla\u_j|)\,|\nabla\u_j|^2\,d\x\\
&\le\limsup_{j\to+\infty}\,\|\omega(h_j\nabla \u_j)\|_{L^\infty(\om)}\int_{\om} |\nabla\u_j|^2\,d\x =0.
  \end{aligned}\end{equation*}
The limit in the last line is zero since $h_j\nabla \u_j\to0$ in $L^\infty(\Omega)$, since $\omega$ is increasing with $\lim_{t\to0^+}\omega(t)\to 0$  and since $\u_j\to\u$ in $H^1(\Omega,\R^3)$ as $j\to+\infty$.
Then, \[\begin{aligned}\limsup_{j\to+\infty} \mathcal |\mathcal G^I_{h_j}(\mathbf y_j)-\mathcal G^I(\u)|&\le 
\limsup_{j\to+\infty} \mathcal |\mathcal G^I_{h_j}(\mathbf y_j)-\mathcal F(\u_j)|+\limsup_{j\to+\infty} \mathcal |\mathcal F(\mathbf u_j)-\mathcal G^I(\u)|\\
&\le\limsup_{j\to+\infty}\left|\frac12\int_\om \nabla\u_j^T D^2\mathcal W(\x,\mathbf I) \nabla\u_j-\frac12\int_\om\nabla\u^T D^2\mathcal W(\x,\mathbf I) \nabla\u\right|\\&\qquad\qquad +\limsup_{j\to+\infty}|\mathcal L(\mathbf R_*\u_j)-\mathcal L(\mathbf R_*\u)|=0
\end{aligned}
\]
where the limit is zero since  $\u_j\to\u$ strongly in $H^1(\Omega,\R^3)$  as $j\to+\infty$. 
\end{proof}
 We next conclude the proof of Theorem \ref{mainth1}  after having recalled that functionals $\mathcal G^I_h$ are uniformly bounded from below, which is a result that is shown in \cite{MP2}.
\begin{lemma}[{\cite[Lemma 4.1]{MP2}}]\label{lemmabound} Assume  \eqref{framind}, \eqref{Z1}, \eqref{reg}, \eqref{coerc}, 
   \eqref{globalequi} and \eqref{comp}.
There exists a constant $C> 0$ (only depending on $\Omega,p,\mathbf f,\mathbf g$) such that
$\mathcal G^I_h(\mathbf y)\ge-C$ for any $h\in (0,1)$ and any $\mathbf y\in W^{1,p}(\Omega,\mathbb R^3)$.
\end{lemma}

\begin{proofth1}
We obtain \eqref{convmin}  from Lemma \ref{lemmabound}. If $(\mathbf y_j)_{j\in\mathbb N}\subset W^{1,p}(\om,\mathbb R^3)$ is a sequence of quasi-minimizers of $\mathcal G^I_{h_j}$,
then by Lemma \ref{compactness} there exists $\u_*\in H^1(\om, \mathbb R^3)$ such that if
 $\mathbf R_j\in\mathcal A_p(\mathbf y_j)$  \KKK and 
$ \u_j(\x):=h_j^{-1}\mathbf R_j^T(\mathbf y_j(\x)-\mathbf R_j\x)$  then, up to subsequences, $\nabla \u_j\wconv \nabla\u_*$ weakly in $L^p(\om)$.  Hence by Lemma \ref{lowerbd} 
\begin{equation*}\lab{liminf}
\displaystyle\liminf_{j\to +\infty}{\mathcal G}^I_{h_j}(\mathbf y_j)\ge {\mathcal G}^I(\u_*).
\end{equation*}
On the other hand, by Lemma \ref{upbd},  for every $\u\in W^{1,p}(\Omega,\mathbb R^3)$ there exist a sequence $(\mathbf u_j)_{j\in\mathbb N}\subset W^{1,p}(\Omega,\mathbb R^3)$ satisfying
$\u_j\wconv \u \hbox { \rm  weakly in}\  W^{1,p}(\Omega,\mathbb R^3)$ as $j\to+\infty$ and   $\mathbf R_*\in \mathcal S^0_{\mathcal L}$ such that by setting $\tilde{\mathbf y}_j:=\mathbf R_*(\x+h_j\u_j)$ we have
\[ \limsup_{j\to +\infty} {\mathcal G}_{h_j}^I(\tilde{\mathbf y}_j) \le {\mathcal G}^I(\u).\]
Since
$${\mathcal G}^I_{h_j}(\mathbf y_j)+o(1)=\inf_{W^{1,p}(\om,\mathbb R^3)}\mathcal G^{I}_{h_{j}}\le {\mathcal G}^I_{h_j}(\tilde{\mathbf y}_j)\qquad \mbox{as $j\to+\infty,$}$$
 by passing to the limit as $j\to +\infty$ we get 
 ${\mathcal G}^I(\u_*)\le {\mathcal G}^I(\u)$ for every $\u\in H^1(\om,\R^3)$ thus completing the proof.
\end{proofth1}

\subsection{The compressible case}
Here we briefly show how to adapt the previous arguments to obtain the proof of Theorem \ref{mainth1comp}.
\begin{lemma}{\bf (Compactness)}.
Assume  \eqref{framind}, \eqref{Z1}, \eqref{reg}, \eqref{coerc}, 
   \eqref{globalequi} and \eqref{comp}.
 Let $(h_{j})_{j\subset\mathbb N}\subset(0,1)$ be a vanishing a sequence, let $M>0$   
 and let $(\mathbf y_{{j}})_{j\in\mathbb N}\subset W^{1,p}(\Omega,\R^3)$ be a sequence  such that  \begin{equation}\lab{boundGcompr}\mathcal G_{h_{j}}(\mathbf y_{j})\le M\qquad\forall \ j\in\mathbb N.\end{equation} 
 Let   $\mathbf R_j\in\mathcal A_p(\mathbf y_j)$ \KKK and $\u_j(\x):=h_j^{-1}(\mathbf R_j^T\mathbf y_j(\x)- \x)$.
  Then, the sequence $(\nabla \u_j)_{j\in\mathbb N}$ is bounded in $L^p(\om,\R^{3\times3})$ and any of its weak $L^p(\om,\R^{3\times3})$ limit points is of the form $\nabla\u_*$ for some $\u_*\in H^1(\om,\R^3)$. Moreover, any limit point of the sequence $(\mathbf R_j)_{j\in\mathbb N}\subset SO(3)$ belongs to $\mathcal S^0_{\mathcal L}$. \KKK
\end{lemma}
\begin{proof} It is readily seen that inequalities \eqref{bdw} and \eqref{odds} holds true with $\mathcal W$ in place of $\mathcal W^I$ hence by arguing as in Lemma \ref{compactness} we get that $\nabla\mathbf u_j$ are equibounded in $L^p$. Moreover  \eqref{boundGcompr} entails the analogous of \eqref{esse0} hence $\mathcal L((\mathbf R_j-\mathbf I)\x)\to 0$ as $j\to+\infty$ and  if  $\mathbf R_j\to \mathbf R_*$ along a suitable subsequence, we have $\mathcal L((\mathbf R_*-\mathbf I)\x)= 0$ that is $\mathbf R_*\in \mathcal S^0_{\mathcal L}$. The remaining part of the proof is identical to that of Lemma \ref{compactness}.
\end{proof}
 \begin{lemma}{\bf (Lower bound)}. Assume \eqref{globalequi}, \eqref{comp}, \eqref{framind}, \eqref{Z1}, \eqref{reg},    
\eqref{coerc}.
 Let $(\mathbf y_j)_{j\in\mathbb N}\subset W^{1,p}(\om,\R^3)$ be a sequence. For any $j\in\mathbb N$, let  $\mathbf R_j\in\mathcal A_p(\mathbf y_j)$ \KKK and  $\u_j(\x):=h_j^{-1}\mathbf R_j^T(\mathbf y_j(\x)-\mathbf R_j\x)$. Suppose that there exists $\u\in W^{1,p}(\om,\R^3)$ such that \KKK
  $\nabla\u_j\wconv \nabla\u$ weakly in $L^p(\om, \mathbb R^3)$. Then
\begin{equation*}
\displaystyle\liminf_{j\to +\infty}\mathcal G_{h_j}(\mathbf y_j)\ge \mathcal G(\u).
\end{equation*}
\end{lemma}
\begin{proof} It is enough to notice that inequality \eqref{lowerbd} in the proof of Lemma \ref{lowerbd} holds true with $\mathcal G_{h_j}$ in place of  $\mathcal G_{h_j}^I$. The proof follows by means of the same arguments therein.
\end{proof}

\begin{lemma}{\bf (Upper bound)}
Assume  \eqref{framind}, \eqref{Z1}, \eqref{reg},    
\eqref{coerc}. Let $(h_j)_{j\in\mathbb N}\subset(0,1)$ be a vanishing sequence. For every $\u\in W^{1,p}(\Omega,\mathbb R^3)$ there exists a sequence $(\mathbf u_j)_{j\in\mathbb N}\subset W^{1,p}(\Omega,\mathbb R^3)$ such that 
$\u_j\wconv \u \hbox { \rm  weakly in}\  W^{1,p}(\Omega,\mathbb R^3)$ as $j\to+\infty$ and $\mathbf R_*\in \mathcal S^0_{\mathcal L}$ such that by setting $\mathbf y_j:=\mathbf R_*(\x+h_j\u_j)$ we have
\[ \limsup_{j\to +\infty} {\mathcal G}_{h_j}(\mathbf y_j) \le {\mathcal G}(\u).\]

\end{lemma}
\begin{proof}
We assume wlog that $\u\in H^1(\om,\mathbb R^3)$. If we let $(\u_j)_{j\in\mathbb N}$ be a sequence obtained by a standard mollification of $\u$, then properties ii) and iii) of Lemma \ref{Kato} hold true. We also let
$$\mathbf R_*\in \argmin\left\{ \int_\Omega \mathcal Q(\x, E(\u))\,d\x-\mathcal L(\mathbf R\u): \mathbf R\in \mathcal S^0_{\mathcal L}\right\},$$
so that by
letting $\mathbf y_j:=\mathbf R(\x+h_j\u_j)$ we obtain
 \begin{equation*}\label{nee}\begin{aligned}
\displaystyle\limsup_{j\to+\infty} \mathcal |\mathcal G_{h_j}(\mathbf y_j)-\mathcal G(\u_j)|
\displaystyle\le \limsup_{j\to+\infty}\int_{\Omega} \left|\frac{1}{h_j^2}\mathcal W(\x,\mathbf I+h_j\nabla\u_j)-\frac12\,\nabla\u_j^T D^2\mathcal W(\x,\mathbf I) \nabla\u_j\right|\,d\x=0
  \end{aligned}\end{equation*}
where the limit is zero by the same argument used in the proof of Lemma \ref{upbd}. Therefore
 \[\begin{aligned}\limsup_{j\to+\infty} \mathcal |\mathcal G_{h_j}(\mathbf y_j)-\mathcal G(\u)|&\le 
\limsup_{j\to+\infty} \mathcal |\mathcal G_{h_j}(\mathbf y_j)-\mathcal \mathcal G(\u_j)|+\limsup_{j\to+\infty} \mathcal |\mathcal G(\mathbf u_j)-\mathcal G(\u)|\\
&\le\limsup_{j\to+\infty}\left|\frac12\int_\om \nabla\u_j^T D^2\mathcal W(\x,\mathbf I) \nabla\u_j-\frac12\int_\om\nabla\u^T D^2\mathcal W(\x,\mathbf I) \nabla\u\right|\\&\qquad\qquad +\limsup_{j\to+\infty}|\mathcal L(\mathbf R_*\u_j)-\mathcal L(\mathbf R_*\u)|=0
\end{aligned}\]
where the limit is zero thanks to the strong convergence of $\u_j$ to $\u$ in $H^1(\om,\mathbb R^3)$.\end{proof}


\begin{proofad2}
By arguing as in Lemma 3.1 of  \cite{MPTARMA} it is readily seen that there exists a constant $C> 0$ (only depending on $\Omega,p,\mathbf f,\mathbf g$) such that
$\mathcal G_h(\mathbf y)\ge-C$ for any $h\in (0,1)$ and any $\mathbf y\in W^{1,p}(\Omega,\mathbb R^3)$. Therefore the proof can be achieved by repeating the argument of the proof of Theorem \ref{mainth1}.
\end{proofad2}

\section{ The gap with linear elasticity: proof of Theorem \ref{mainth3}}\label{counterexamples}
 In this section we show that if $\mathcal S^0_{\mathcal L}$ is not reduced to the identity matrix,
 the minimization problem in the definition of functional $\mathcal G^I(\u)$ (resp. $\mathcal G(\u)$) is not solved in general by $\mathbf R=\mathbf I$ if $\u$ minimizes $\mathcal G^I$ (resp. $\mathcal G$) over $W^{1,p}(\Omega,\R^3)$. In particular, the minimal value of $\mathcal G^I$  (resp. $\mathcal G$) can be strictly below the minimum value of the standard functional of linearized elasticity $\mathcal E^I$ (resp. $\mathcal E$). 

 Here and in the following of this section, $\Omega$ is the set defined in \eqref{cyl2}, $B$ denotes the unit ball centered at the origin in the $xy$ plane, while $\nabla$ and $\Delta$ shall denote the gradient and the Laplacian in the $x,y$ variables, respectively. 
 Moreover, the form of external forces is that of \eqref{ef2}, and the conditions \eqref{effe12}-\eqref{effe22} are assumed to hold.
We also introduce the auxiliary volume force field $$\tilde {\mathbf f}(x,y,z):=(\varphi_y(x,y),-\varphi_x(x,y),\psi(z)),$$ 
and we notice that  
 $\tilde{\mathbf f}=\tilde{\mathbf R}\mathbf f$, where $\tilde{\mathbf R}$ is the rotation matrix
 \begin{equation}\label{tildematrix}
\tilde{\mathbf R}:={\footnotesize\left(\begin{array}{ccc}0&1&0\\-1&0&0\\0&0&1\end{array}\right)}.
 \end{equation}

 Exploiting \eqref{effe12} and \eqref{effe22}, it is not difficult to check that  $\mathcal L$ from \eqref{ef2} satisfies \eqref{globalequi}. Concerning \eqref{comp},
 in view of the general form of $\mathbf W\in\mathbb R^{3\times3}_{\mathrm{Skew}}$, i.e.,
 \begin{equation*}{\footnotesize
 \mathbf W:=\left(\begin{array}{ccc}0&a&b\\-a&0&c\\-b&-c&0\end{array}\right)}\qquad a,b,c\in\mathbb R,
 \end{equation*}
  under assumptions \eqref{effe12} and \eqref{effe22} we have
  \[
   \int_\Omega\mathbf f(\x)\cdot\mathbf W\x\,d\x=\int_\Omega\tilde{\mathbf f}(\x)\cdot\mathbf W\x\,d\x=0
  \]
  and
 \begin{equation}\label{middle}
 \int_\Omega\mathbf f(\x)\cdot\mathbf W^2\x\,d\x=\int_\Omega\tilde{\mathbf f}(\x)\cdot\mathbf W^2\x\,d\x=-\pi(b^2+c^2)\int_0^1 z\psi(z)\,dz\le 0, \end{equation}
so that by invoking the Euler-Rodrigues formula \eqref{eurod}  we see that $\mathcal L$ satisfies \eqref{comp}  as well.
 From \eqref{middle} we see that  if  $\int_0^1 z\psi(z)\,dz>0$, then 
 \[
\int_\Omega \mathbf f(\x)\cdot\mathbf W^2\x\,d\x = 0
\]
 if and only if   $b^2+c^2=0$, so that in view of \eqref{eurod} the set  $\mathcal S^0_{\mathcal L}$ coincides with the set of rotation matrices around the $z$ axis, i.e., 
  \begin{equation}\label{S0Lcontrex}
  \mathcal S^0_{\mathcal L}=\{\mathbf R_\theta:\theta\in[-\pi,\pi]\},\qquad\mbox{where}\quad
\mathbf R_\theta={\footnotesize\left(\begin{array}{ccc}\cos\theta&\sin\theta&0\\-\sin\theta&\cos\theta&0\\0&0&1\end{array}\right)},
 \end{equation}
 in particular $\mathcal S^0_{\mathcal L}$ is not reduced to the identity matrix and it is a strict subset of $SO(3)$.  
 On the other hand, if $\int_0^1z\psi(z)\,dz=0$, then  $\mathcal S_{\mathcal L}^0\equiv SO(3)$.  
In both cases we clearly have
\begin{equation}\label{inthekernel}
\tilde{\mathbf R}\in\mathcal S^0_{\mathcal L}\qquad\mbox{and}\qquad \tilde{\mathbf R}^T\in\mathcal S^0_{\mathcal L}.
\end{equation}

 \KKK

 Concerning the strain energy density, in this section we assume that $\mathcal W$ satisfies  \eqref{framind}, \eqref{Z1}, \eqref{reg}, \eqref{coerc} and \eqref{Venant3},
 and we let $\mathcal W^I(\x,\mathbf F)=\mathcal W^I(\mathbf F)$ be equal to $\mathcal W(\mathbf F)$ is $\det\mathbf F=1$ and equal to $+\infty$ otherwise.
 With these assumptions on $\mathcal W$, 
 the  functional of linearized elasticity is reduced to
 \[
 \mathcal E(\u):=4\int_{\Omega}|\mathbb E(\u)|^2\,d\x-\mathcal L(\u),\qquad \u\in H^1(\om,\R^3),
 \] 
  while the limit functional $\mathcal G$ becomes
 \begin{equation*}
\displaystyle {\mathcal G}(\u)= 4\int_\om | \mathbb E(\u)|^2\,d\x-\mathcal L(\u)-\max_{\mathbf R\in\mathcal S^0_{\mathcal L}}\mathcal L((\mathbf R-\mathbf I)\x),\qquad  \u\in H^1(\om,\mathbb R^3).
\end{equation*}
For the following arguments, it is also convenient to introduce the auxiliary functional
\begin{equation}\label{auxfunct}
\tilde{\mathcal G}(\u):=\displaystyle4\int_{\Omega}|\mathbb E(\u)|^2\,d\x-\mathcal L(\tilde{\mathbf R}^T\u),\qquad \u\in H^1(\om,\mathbb R^3)\end{equation}
where $\tilde {\mathbf R}$ is given by \eqref{tildematrix}. \MMM The above functionals $\mathcal E$, $\mathcal G$, $\tilde{\mathcal G}$ are extended as usual to $W^{1,p}(\om,\R^3)\setminus H^1(\om,\R^3)$ with value $+\infty$. \KKK
 Due to \eqref{globalequi} and Korn inequality, it follows from standard arguments that the functionals $\mathcal E,\mathcal G,\tilde{\mathcal G}$   admit minimizers over $H^1(\Omega,\mathbb R^3)$. 

 We shall also consider the incompressible case, by considering functionals  $\mathcal G^I$, $\mathcal E^I$, as defined in Section \ref{sectmain}, \MMM under the assumption \eqref{Venant3} \KKK for $\mathcal W$, and with external loads given by functional $\mathcal L$ from \eqref{external} with $\mathbf g\equiv0$ and $\mathbf f$ of the form \eqref{ef2}. Again, we introduce the auxiliary functional 
\[\tilde{\mathcal G}^I(\u):=\left\{\begin{array}{ll}
\displaystyle4\int_{\Omega}|\mathbb E(\u)|^2\,d\x-\mathcal L(\tilde{\mathbf R}^T\u)\quad&\hbox{if} \ \u\in H^1_{\dv}(\om,\mathbb R^3)\vspace{0.1cm}\\
+\infty\quad &\hbox{otherwise in} \ \MMM W^{1,p}(\om,\mathbb R^3)\KKK.
\end{array}\right.\]
where $\tilde {\mathbf R}$ is given by \eqref{tildematrix}. Existence of minimizers also holds for $\mathcal E^I,\mathcal G^I,\tilde {\mathcal G}^I$, since  the divergence-free constraint is weakly closed in $H^1(\om,\mathbb R^3)$.

 Before proving Theorem \ref{mainth3}, we provide an auxiliary statement.

%

 \begin{lemma}\label{lemma21} 
Assume \eqref{cyl2}, \eqref{ef2}, \eqref{Venant3}, \eqref{effe12}, \eqref{effe22}.
Then there hold
\[
\min_{H^1(\Omega,\mathbb R^3)}\tilde{\mathcal G}\le \min_{u\in H^2(B)} \int_{B}8u_{xy}^2+2(u_{yy}-u_{xx})^2+ u \Delta\varphi \;\;+\min_{w\in H^1(0,1)} 4\pi\int_0^1 w'^{\,2}-\pi\int_{0}^1 w\psi<0\]
and
\[
\min_{H^1(\Omega,\mathbb R^3)}\tilde{\mathcal G}^I\le \min_{u\in H^2(B)} \int_{B}8u_{xy}^2+2(u_{yy}-u_{xx})^2+ u \Delta\varphi <0.
\]
 \end{lemma}
 \begin{proof} Let us prove the first statement.
 Let $\mathcal K\subset H^1(\Omega,\mathbb R^3)$ be the class of displacement fields $\v\in H^1(\Omega,\mathbb R^3)$ of the form \begin{equation}\label{K}\v(x,y,z)=(u_y(x,y),-u_x(x,y), w(z)),\end{equation} for some $u\in H^2(\Omega)$ and some $w\in H^1(0,1)$. 
 It is easily seen that $\mathcal K$ is weakly closed in $H^1(\Omega,\mathbb R^3)$. 
 In particular, there are minimizers of $\tilde{\mathcal G}$ over $\mathcal K$.
  We notice that by means of \eqref{K} any couple $u\in H^2(B)$, $w\in H^1(0,1)$ uniquely determines $\v\in \mathcal K$. Conversely,    any $\v\in \mathcal K$ uniquely determines $u\in H^2(B)$, up to an additive constant, and $w\in H^1(0,1)$. Moreover a computation shows that the energy functional $\tilde{\mathcal G}$ takes the following form for any $\v\in\mathcal K$:
 \begin{equation*}\label{form}\begin{aligned}
 \tilde{\mathcal G}(\v)&=\tilde{\mathcal G}((u_y,-u_x,w))=\int_{B}8u_{xy}^2+2(u_{yy}-u_{xx})^2-\int_B (u_y,-u_x)\cdot(\varphi_y,-\varphi_x)\\&
 \qquad+4\pi\int_0^1 w'^{\,2}\,dz-\pi\int_{0}^1 w\psi\,dz\\&
 =\int_{B}8u_{xy}^2+2(u_{yy}-u_{xx})^2+\int_B u \Delta\varphi+4\pi\int_0^1 w'^{\,2}\,dz-\pi\int_{0}^1 w\psi\,dz. \end{aligned}
 \end{equation*}
   having exploited the fact that $\nabla\varphi\cdot \mathbf n=0$ on $\partial B$, where $\mathbf n$ is the outer unit normal to $\partial B$. We deduce
 \begin{equation}\label{comparison}\begin{aligned}
 \min_{H^1(\Omega,\mathbb R^3)}\tilde{\mathcal G}\le \min_{\mathcal K}\tilde{\mathcal G}&=\min_{u\in H^2(B)} \int_{B}8u_{xy}^2+2(u_{yy}-u_{xx})^2+\int_B u \Delta\varphi \\&\qquad+\min_{w\in H^1(0,1)} 4\pi\int_0^1 w'^{\,2}-\pi\int_{0}^1 w\psi,\end{aligned}
 \end{equation}
 where the last minimization problem in \eqref{comparison} has a solution as well, thanks to $\int_\Omega\psi=0$ from \eqref{effe22}  and to Poincar\'e inequality.
  
 If we introduce the perturbation $u+\eps\zeta$, where $\eps>0$ and $\zeta\in H^2_0(B)$, we get the first order optimality condition for the first minimization problem in the right hand side of \eqref{comparison}
 \begin{equation*}\label{variation}
 4\int_B 4u_{xy}\zeta_{xy}+u_{yy}\zeta_{yy}+u_{xx}\zeta_{xx}-u_{yy}\zeta_{xx}-u_{yy}\zeta_{xx}=-\int_B\zeta \Delta\varphi, 
 \end{equation*}
 and after integration by parts we obtain the Euler-Lagrange equation $$4\Delta^2 u+\Delta\varphi=0\qquad\mbox{ in $\mathcal D'(B)$},$$
 where $\Delta^2$ denotes the planar biharmonic operator.
 
 In view of \eqref{comparison}, and since $\tilde {\mathcal G}(\mathbf 0)=0$, in order to conclude it is enough to show that $\min_{\mathcal K}\tilde{\mathcal G}\neq0$.
 Assume by contradiction 
  that $\min_{\mathcal K}\tilde{\mathcal G}=0$: then $u\equiv 0$ on $B$, $w\equiv 0$ on $(0,1)$ are  solutions  to the minimization problems in the right hand side of \eqref{comparison}, 
  so that $\Delta^2u\equiv0$ in $B$ and from above the Euler-Lagrange equation we deduce $\Delta\varphi\equiv 0$ in $B$. This is a contradiction with assumption \eqref{effe12}.

In order to prove the second statement, we consider the subset $\mathcal K_{\dv}$ of $H^1_{\dv}(\Omega,\mathbb R^3)$ made of vector fields of the form $\v(x,y,z)=(u_y(x,y), -u_x(x,y),0)$ for some $u\in H^2(B)$, and for any $\v\in \mathcal K_{\dv}$ the energy functional $\tilde{\mathcal G}^I$ has the expression
 \begin{equation*}\begin{aligned}
 \tilde{\mathcal G}^I(\v)&=\tilde{\mathcal G}^I((u_y,-u_x,0))
 =\int_{B}8u_{xy}^2+2(u_{yy}-u_{xx})^2+\int_B u \Delta\varphi, \end{aligned}
 \end{equation*}
so that
 \begin{equation*}\begin{aligned}
 \min_{H^1(\Omega,\mathbb R^3)}\tilde{\mathcal G}^I\le \min_{\mathcal K_{\dv}}\tilde{\mathcal G}^I&=\min_{u\in H^2(B)} \int_{B}8u_{xy}^2+2(u_{yy}-u_{xx})^2+\int_B u \Delta\varphi 
 \end{aligned}
 \end{equation*}
and the proof concludes 
with the same argument as above.
 \end{proof}
 
 We proceed to the proof of Theorem \ref{mainth3}.
%
 \begin{proofad3}
We will prove that there holds $$\min_{H^1(\om,\R^3)}\mathcal G\le \min_{H^1(\Omega,\mathbb R^3)}\tilde{\mathcal G}<\min_{H^1(\Omega,\mathbb R^3)}{\mathcal E},$$
and that if $\|\psi\|_{L^2(\om)}$ is small enough there also holds \KKK
$$\min_{H^1(\om,\R^3)}{\mathcal G}^I\le\min_{H^1(\Omega,\mathbb R^3)}\tilde{\mathcal G}^I< \min_{H^1(\Omega,\mathbb R^3)}\mathcal E^I.$$

 We start by noticing that the first inequality in both statements is trivial, due to \eqref{inthekernel}. Therefore, we are left to prove the second inequality in both cases.
 
Let $\v=(v_1,v_2,v_3)\in H^1(\Omega,\mathbb R^3)$. Let $\tilde \v:=(v_1,v_2)$, so that $\tilde \v\in H^1(\Omega,\mathbb R^2)$.   Let  $\tilde \u\in H^1(B,\mathbb R^2)$ be defined by $\tilde\u(x,y):=\int_0^1 \tilde\v(x,y,z)\,dz$. Let moreover $\tilde w\in H^1(0,1)$ be defined by $\tilde w(z):=\tfrac1\pi\int_B v_3(x,y,z)\,dx\,dy$. Since $|\mathbb E(\v)|\ge |\widetilde{\mathbb E}(\v)|=|\widetilde{\mathbb E}(\tilde \v)|$, where $\widetilde{\mathbb E}(\cdot)$ is the upper-left $2\times2$ submatrix of $\mathbb E(\cdot)$, by Jensen inequality we have \begin{equation}\label{long}\begin{aligned}
\mathcal E^I(\v)&\ge\mathcal E(\v)\ge 4\int_\Omega|\mathbb E(\v)|^2-\int_\Omega (\varphi_x,\varphi_y,\psi)\cdot\v\\&\ge
4\int_B\left|\int_0^1\widetilde{\mathbb E}(\tilde \v)\,dz\right|^2\,dx\,dy
+4\pi\int_0^1\left|\frac1\pi\int_B v_{3,z}\,dx\,dy\right|^2\,dz\\&\qquad
-\int_B (\varphi_x,\varphi_y)\cdot\left(\int_0^1\tilde \v\,dz\right)\,dx\,dy-\int_0^1\psi(z)\left(\int_B v_3(x,y,z)\,dx\,dy\right)\,dz
\\&\ge\mathcal J(\tilde\u)+4\pi\int_0^1\tilde w'{\,^2}-\pi\int_0^1 \tilde w\psi,
\end{aligned}\end{equation}
where we have used the notation  $v_{i,a}:=\partial_av_i$, $i\in\{1,2,3\}$, $a\in\{x,y,z\}$,
and where
\begin{equation}\label{Jey}{\mathcal J(\widetilde \u)}:=4\int_B|\widetilde{\mathbb E}(\tilde \u)|^2-\int_B\nabla\varphi\cdot\tilde \u.\end{equation}

 We claim that functional $\mathcal J$ admits a minimizer over $H^1(B,\mathbb R^2)$ which is the gradient of a $H^2(B)$ function.
Indeed, thanks to \eqref{effe12} and Korn inequality, a minimizer exists and it is unique up to  planar  infinitesimal rigid displacements. Moreover, thanks to a first variation argument, it is solution
  to the boundary value problem 
\begin{equation}\label{elasticproblem}
\left\{\begin{array}{ll}-8\dv\widetilde{\mathbb{E}}(\tilde\u)=\nabla\varphi\qquad&\mbox{in $B$}\\
\widetilde{\mathbb E}(\tilde \u)\mathbf n=0\qquad&\mbox{on $\partial B$}.
\end{array}\right.
\end{equation}
We claim that problem \eqref{elasticproblem} admits a solution
 $\tilde \u\in H^1(B,\mathbb R^2)$  of the form $\tilde\u(r) := r^{-1}\eta(r)(x,y)$, where $  r:=\sqrt{x^2+y^2}$ (since $\tilde \u\in H^1(B,\mathbb R^2)$, we necessarily have $\eta(0)=0$).
 Indeed, a direct calculation shows that $\tilde\u(r)= r^{-1}\eta(r)(x,y)$ is a $H^1(B,\mathbb R^2)$ solution to  problem \eqref{elasticproblem} if and only if 
\begin{equation}\label{ode}\left\{\begin{array}{ll}r^2\eta''+r\eta'-\eta=-\frac18\,r^2\phi'(r)\qquad &\mbox{in $(0,1)$}\\
\eta(0)=0\\
\eta'(1)=0\\
\end{array}\right. 
\end{equation}
where $\phi$ is the radial profile of $\varphi$. Problem \eqref{ode} has indeed a solution, whose explicit form is
\begin{equation}\label{etastar}
\eta_*(r)=-\frac1{16}\,r\phi(r)+\frac1{16r}\int_0^r t^2\phi'(t)\,dt.
\end{equation}
Therefore, letting $\Phi:B\to\mathbb R$ be the radial function defined by $\Phi(x,y):=\int_0^r\eta_*(t)\,dt$,  we obtain $\Phi\in H^2(B)$ and moreover $\nabla \Phi=r^{-1}\eta(r)(x,y)$ solves \eqref{elasticproblem}, hence it minimizes $\mathcal J$ over $H^1(\Omega,\mathbb R^2)$. The claim is proved.

\KKK


 Therefore, the  estimate \eqref{long} rewrites as
 \[\begin{aligned}
\mathcal E^I(\v)\ge\mathcal E(\v)&\ge \mathcal J(\tilde\u)
+4\pi\int_0^1\tilde w'{\,^2}-\pi\int_0^1 \tilde w\psi\\
&\ge \min_{\Phi\in H^2(B)} 4\int_B|\widetilde{\mathbb E}(\nabla\Phi)|^2-\int_B\nabla\Phi\cdot\nabla\varphi+4\pi\int_0^1\tilde w'{\,^2}-\pi\int_0^1 \tilde w\psi.
\end{aligned}\]
Integrating by parts, since \eqref{effe12} yields $\phi=0$ on $\partial B$, we get
\[
\mathcal E^I(\v)\ge\mathcal E(\v)\ge \min_{\Phi\in H^2(B)}4\int_B|D^2\Phi|^2+\int_B\Phi\,\Delta\varphi+4\pi\int_0^1\tilde w'{\,^2}-\pi\int_0^1 \tilde w\psi,
\] \KKK
where $D^2$ denotes the Hessian in the $x,y$ variables,
hence \begin{equation}\label{hence}\begin{aligned}\min_{H^1(\Omega,\mathbb R^3)}\mathcal E^I\ge \min_{H^1(\Omega,\mathbb R^3)}\mathcal E&\ge 
 \min_{\Phi\in H^2(B)}\int_B 8\Phi_{xy}^2+4\Phi_{xx}^2+4\Phi_{yy}^2+\int_B\Phi\Delta\varphi\\&\qquad+\min_{w\in H^1(0,1)}4\pi\int_0^1w'^{\,2}-\pi\int_\Omega w\psi.
 \end{aligned}\end{equation}\KKK

Suppose that $\tilde\Phi\in H^2(B)$ is a solution to the first minimization problem on the right hand side of \eqref{hence}.
We have already proven that $\nabla\tilde\Phi$ solves \eqref{elasticproblem}, and taking the divergence therein shows that $\tilde\Phi$ solves the biharmonic equation
$8\Delta^2\Phi=-\Delta\varphi$ in $B$.  \KKK  
As \eqref{effe12} requires that $\Delta\varphi$ is not identically zero on $B$, we deduce that $\Delta\tilde \Phi$ is not identically zero as well.
This implies by Young inequality 
\begin{equation}\label{young}
\int_B 8\tilde \Phi_{xy}^2+2(\tilde \Phi_{yy}-\tilde\Phi_{xx})^2+\int_B\tilde\Phi\Delta\varphi<\int_B  8\tilde\Phi_{xy}^2+4\tilde\Phi_{xx}^2+4\tilde\Phi_{yy}^2-\int_B\tilde\Phi\Delta\varphi.
\end{equation}
Notice that the inequality is strict, since the Young inequality $2(\tilde \Phi_{yy}-\tilde\Phi_{xx})^2\le 4\tilde\Phi_{xx}^2+4\tilde\Phi_{yy}^2$ holds with equality if and only if $\tilde \Phi_{xx}=-\Tilde \Phi_{yy}$, and we have just checked that $\Delta\tilde \Phi$ does not vanish identically on $B$. In particular, $2(\tilde \Phi_{yy}-\tilde\Phi_{xx})^2< 4\tilde\Phi_{xx}^2+4\tilde\Phi_{yy}^2$ on a set of positive measure in $B$.
 From Lemma \ref{lemma21}, from \eqref{young} and \eqref{hence} we infer 
\[\begin{aligned}\min_{H^1(\Omega,\mathbb R^3)}\tilde{\mathcal G}&
\le \min_{u \in H^2(B)} \int_B 8u_{xy}^2+2(u_{yy}-u_{xx})^2+\int_Bu\Delta\varphi+\min_{w\in H^1(0,1)} 4\pi\int_0^1 w'^{\,2}-\pi\int_{0}^1 w\psi\\&
\le \int_B 8\tilde\Phi_{xy}^2+2(\tilde\Phi_{yy}-\tilde\Phi_{xx})^2+\int_B\tilde\Phi\Delta\varphi+\min_{w\in H^1(0,1)} 4\pi\int_0^1 w'^{\,2}-\pi\int_{0}^1 w\psi\\&< \int_B  8\tilde\Phi_{xy}^2+4\tilde\Phi_{xx}^2+4\tilde\Phi_{yy}^2+\int_B\tilde\Phi\Delta\varphi+\min_{w\in H^1(0,1)} 4\pi\int_0^1 w'^{\,2}-\pi\int_{0}^1 w\psi\\&
= \min_{\Phi\in H^2(B)}\int_B  8\Phi_{xy}^2+4\Phi_{xx}^2+4\Phi_{yy}^2+\int_B\Phi\Delta\varphi+\min_{w\in H^1(0,1)} 4\pi\int_0^1 w'^{\,2}-\pi\int_{0}^1 w\psi\\&\le\min_{ H^1(\Omega,\mathbb R^3 )}\mathcal E
\le\min_{ H^1(\Omega,\mathbb R^3 )}\mathcal E^I
\end{aligned}\]
thus concluding the proof of the first statement.

%

Let us now prove the second statement, concerning the incompressible case.{
\KKK
Let
$$C(\varphi,\psi):=\|\psi\|_{L^2(\om)}\left (\int_\om |\nabla\varphi|^2+|\psi|^2\right )^{1/2}.$$
 If $\v^*=(v_1^*,v_2^*,v_3^*)$ minimizes ${\mathcal E}^I$ over $H^1(\om,\R^3)$,
 then we have 
 \[
0= \frac{d}{d\eps}{\Bigg|}_{\eps=0}\mathcal E^I((1+\eps)\v^*)=8\int_\Omega |\mathbb E(\v^*)|^2-\mathcal L(\v^*),
 \]
  and then by applying  \eqref{globalequi} and H\"older inequality
  \[
  8\int_\om|\mathbb E(\v^*)|^2\,dx=\mathcal L(\v^*)=\mathcal L(\v^*-\mathbb P\v^*)\le \|\mathbf f\|_{L^{6/5}(\om,\mathbb R^3)}\|\v^*-\mathbb P\v^*\|_{L^6(\om,\R^3)}.
  \]
  By \eqref{kornvera} and H\"older inequality again we deduce therefore
\begin{equation}\lab{estmin}
\left (\int_\om |\mathbb E(\v^*)|^2\right )^{1/2}\le K_1\|\mathbf f\|_{L^2(\om,\R^3)}\le K_1 \left (\int_\om |\nabla\varphi|^2+|\psi|^2\right )^{1/2}
\end{equation}
for some suitable constant $K_1$ (only depending on $\Omega$).
 By taking into account \eqref{estmin}, still by Korn and H\"older inequality  and by \eqref{effe22} we get
\begin{equation}\lab{vertlav}\begin{aligned}
 \left |\int_\om \psi v^*_3\right |&= \left |\int_\om \psi (\v^*-\mathbb P\v^*)_3\right |\le \|\psi\|_{L^2(\om)}\|\v^*-\mathbb P\v^*\|_{L^2(\om)}\\&\le  K_2 \|\psi\|_{L^2(\om)}  \|\mathbb E(\v^*)\|_{L^2(\om)}\le  K\,C(\varphi,\psi),\end{aligned}
\end{equation}
where $ K_2$ is  another constant  that depends only on $\Omega$ and $K=K_1K_2$.
By taking into account 
 \eqref{vertlav}, we get
\begin{equation*}\begin{aligned}
\min_{H^1(\Omega,\mathbb R^3)}\mathcal E^I\ge 4\int_\om |v^*_{1,x}|^2+|v^*_{2,y}|^2+ \frac{1}{2}(v^*_{1,y}+v^*_{2,x})^2
 -\int_\om (v^*_1\varphi_x+v^*_2\varphi_y)- KC(\varphi,\psi)
  \end{aligned}\end{equation*}
 and by Jensen inequality
 \begin{equation*}\begin{aligned}
\min_{H^1(\Omega,\mathbb R^3)}\mathcal E^I&\ge 4\int_B (|\tilde v^*_{1,x}|^2+|\tilde v^*_{2,y}|^2+ \frac{1}{2}(\tilde v^*_{1,y}+\tilde v^*_{2,x})^2
 -\int_B (\tilde v^*_1\varphi_x+\tilde v^*_2\varphi_y)- KC(\varphi,\psi) \\&=\mathcal J(\tilde\v^*)-KC(\varphi,\psi)
 \end{aligned}\end{equation*}
where we have set $\tilde\v^*(x,y):= \int_0^1\v^*(x,y,z)\,dz$ and where $\mathcal J$ is defined by \eqref{Jey}. By  repeating the argument used in the compressible case (and letting $\tilde\Phi$ be defined in the same way)  we obtain 
\begin{equation}\label{lab}\begin{aligned}
\min_{H^1(\Omega,\mathbb R^3)}\mathcal E^I&\ge \left(\min_{\Phi\in H^2(B)}\int_B 8\Phi_{xy}^2+4\Phi_{xx}^2+4\Phi_{yy}^2
+\int_B\Phi\Delta\varphi\right) -KC(\varphi,\psi)\\
&=
\int_B 8\tilde\Phi_{xy}^2+4\tilde\Phi_{xx}^2+4\tilde\Phi_{yy}^2 +\int_B\tilde\Phi\Delta\varphi -KC(\varphi,\psi)\\
&=\int_B 8\tilde\Phi_{xy}^2+2(\tilde\Phi_{yy}-\tilde\Phi_{xx})^2+2|\Delta\tilde\Phi|^2+\int_B \tilde\Phi\Delta\varphi- KC(\varphi,\psi).
\KKK
\end{aligned}
\end{equation}
By \eqref{lab} and by the second statement  of Lemma \ref{lemma21} we deduce
\[
\min_{H^1(\Omega,\mathbb R^3)}{\mathcal E}^I\ge\int_B 2|\Delta\tilde\Phi|^2- KC(\varphi,\psi)+\min_{H^1(\Omega,\mathbb R^3)}\tilde{\mathcal G}^I.
\]
Since $\int_B|\Delta\tilde\Phi|^2>0$ as previously observed,  we may choose $\|\psi\|_{L^2(\Omega)}$ so small  that $KC(\varphi,\psi)<\int_B2|\Delta\tilde\Phi|^2$ and deduce $$\min_{H^1(\Omega,\mathbb R^3)}{\mathcal E}^I>\min_{H^1(\Omega,\mathbb R^3)}\tilde{\mathcal G}^I$$
thus completing the proof.}
 \end{proofad3}

 We conclude by providing some more properties for the compressible case that are directly deduced by refining the arguments in the proof of Theorem \ref{mainth3}, under the further assumption $\int_0^1z\psi(z)\,dz>0$.
First we check that minimizers of the limit functional $\mathcal G$ are not unique up to infinitesimal rigid displacements, in the sense that there are two minimizers whose difference is not an   infinitesimal rigid displacement. In a second statement we obtain a solution of  the problem $\min\mathcal G$, thus showing more explicitly the gap between $\min\mathcal G$ and $\min\mathcal E$.
 
\begin{proposition}\label{pro52} Assume \eqref{cyl2}, \eqref{ef2}, \eqref{Venant3}, \eqref{effe12}, \eqref{effe22} and  $\int_0^1z\psi(z)\,dz>0$. If $\u\in\mathrm{argmin}_{H^1(\om,\R^3)}\mathcal G$, then $\u$ 
is of the form $\u=\w^0+\w$, where $\mathbb E(\w^0)\equiv0$ and $\w=(w_1,w_2,w_3)$ is such that $w_1,w_2$ do not depend on $z$ and $w_3$ does not depend on $x,y$. Moreover, minimizers of $\mathcal G$ over $H^1(\om,\R^3)$
are not unique up to infinitesimal rigid displacements. 
 \end{proposition}
 
 \begin{proof}
 Suppose that  $\u=(u_1,u_2,u_3)\in H^1(\om,\mathbb R^3)$ is a minimizer for functional $\mathcal G$ and let $\bar \u:=(\bar u_1,\bar u_2,\bar u_3)$, where
 \[
 \bar u_1:=\int_0^1u_1(x,y,z)\,dz,\qquad\bar u_2:=\int_0^1 u_2(x,y,z)\,dz,\qquad \bar u_3:=\frac1\pi\int_B u_3(x,y,z)\,dx\,dy.
 \]
Since $\int_0^1z\psi(z)\,dz>0$, the set $\mathcal S_{\mathcal L}^0$ is given by \eqref{S0Lcontrex}, and then it is immediate to check that, due to the specific form of $\mathbf f$ from \eqref{ef2}, there holds
\begin{equation}\label{rsign}
\int_\om\mathbf f\cdot \mathbf R\u\,d\x=\int_\om\mathbf f\cdot \mathbf R\bar\u\,d\x\qquad\forall \ \mathbf R\in\mathcal S_{\mathcal L}^0. 
\end{equation}
Moreover, by applying Jensen inequality, similarly to the proof of Theorem \ref{mainth3}, we obtain
\begin{equation}\label{fourj}\begin{aligned}
&\int_\om u_{1,x}^2\ge\int_\om\bar u_{1,x}^2,\qquad\int_\om u_{2,y}^2\ge\int_\om\bar u_{2,y}^2,\\&\int_\om u_{3,z}^2\ge\int_\om\bar u_{3,z}^2,\qquad\int_\om (u_{1,y}+u_{2,x})^2\ge\int_\om(\bar u_{1,y}+\bar u_{2,x})^2
\end{aligned}\end{equation}
so that
\[\begin{aligned}4\int_\Omega|\mathbb E(\u)|^2\,d\x&=\int_\om \left(4u_{1,x}^2+4u_{2,y}^2+4u_{3,z}^2+(u_{1,y}+u_{2,x})^2+(u_{1,z}+u_{3,x})^2+(u_{2,z}+u_{3,y})^2\right)\\&\ge\int_\om \left(4\bar u_{1,x}^2+4\bar u_{2,y}^2+4\bar u_{3,z}^2+(\bar u_{1,y}+\bar u_{2,x})^2\right)= 4\int_\om |\mathbb E(\bar\u)|^2\,d\x.\end{aligned}\]
But  the latter inequality and the inequalities \eqref{fourj} are necessarily equalities, otherwise in view of \eqref{rsign} we would deduce $\mathcal G(\bar\u)<\mathcal G(\u)$, contradicting minimality of $\u$. This implies
\begin{equation}\label{xyz}
u_{1,z}+u_{3,x}=   u_{2,z}+u_{3,y}\equiv0
\end{equation}
along with the fact that
 $u_{3,z}$ does not depend on $x,y$, since the Jensen inequality
\[
\frac1\pi\int_B(u_{3,z})^2\,dx\,dy\ge \left(\frac1\pi\int_B u_{3,z}\,dx\,dy\right)^2=\bar u_{3,z}^2
\]
is strict unless $u_{3,z}$ is independent of $x,y$. Similarly, we deduce that $u_{1,x}$, $u_{2,y}$ and $u_{1,y}+u_{2,x}$ do not depend on $z$, therefore there exist functions $H=H(z)$, $T=T(x,y)$, $A=A(y,z)$, $B=B(x,y)$, $C=C(x,z)$ and $D=D(x,y)$ such that $u_1,u_2,u_3$ have the form
\begin{equation}\label{tree}\begin{aligned}
u_1(x,y,z)&=A(y,z)+B(x,y),\\
u_2(x,y,z)&=C(x,z)+D(x,y),\\
u_3(x,y,z)&=H(z)+T(x,y),
\end{aligned}
\end{equation} 
and there holds
\begin{equation}\label{AC}
\partial_z(u_{1,y}+u_{2,x})=A_{yz}(y,z)+C_{xz}(x,z)\equiv0.
\end{equation}
Taking \eqref{xyz} into account we deduce \begin{equation}\label{0=}0\equiv A_z(y,z)+T_x(x,y)=C_z(x,z)+T_y(x,y)\end{equation} so that $A_{zz}=C_{zz}=T_{xx}=T_{yy}\equiv0$
and $$0\equiv A_{yz}(y,z)+T_{xy}(x,y)=C_{xz}(x,z)+T_{xy}(x,y).$$ The latter entails, thanks to \eqref{AC}, $A_{yz}= C_{xz}= T_{xy}\equiv 0$.
We conclude that $T$ is a linear function of $x,y$, i.e., $T(x,y)=ax+by+c$ for some real constants $a,b,c$, and then from \eqref{0=} we deduce that there are functions $Q=Q(y)$ and $S=S(x)$ such that $A(y,z)=-az+Q(y)$ and $C(x,z)=-bx+S(x)$. Substituting in \eqref{tree} we have
\begin{equation*}\begin{aligned}
u_1(x,y,z)&=-az+B(x,y)+Q(y),\\
u_2(x,y,z)&=-bz+D(x,y)+S(x),\\
u_3(x,y,z)&=H(z)+c+ax+by,\end{aligned}
\end{equation*} 
where $\mathbb E(-az,-bz,c+ax+by)=0$. 
This shows that if $\u\in H^1(\om,\R^3)$ minimizes $\mathcal G$, then up to adding an infinitesimal rigid displacement $u_1,u_2$  depend only on $x,y$ and $u_3$  depends only on $z$.

Let now $\u^*=(u_1^*,u_2^*,u_3^*)\in H^1(\om,\mathbb R)^3$ be a minimizer of $\mathcal G$ such that $u^*_1,u^*_2$ do not depend on $z$ and $u_3^*$ does not depend on $x,y$, so that by defining $\hat \u:=(-u^*_1,-u^*_2,u^*_3)$ we get $|\mathbb E(\u^*)|^2\equiv |\mathbb E(\hat \u)|^2$.  Let $\mathbf R^*\in{\mathrm{argmax}}_{{\mathcal S}_{\mathcal L}^0} \mathcal L(\mathbf R\u^*)$.   By letting $\hat {\mathbf R}:=\mathrm{diag}(-1,-1,1)\,\mathbf R^*$, we have $\hat{\mathbf R}\in\mathcal S_{\mathcal L}^0$ and
we get 
\[
\min_{H^1(\om,\R^3)}\mathcal G=\int_\om|\mathbb E(\u^*)|^2\,d\x-\mathcal L(\mathbf R^*\u^*)=\int_\om|\mathbb E(\hat\u)|^2\,d\x-\mathcal L(\hat{\mathbf R}\hat\u),
\] 
thus showing that $\hat\u$ is also a minimizer of $\mathcal G$. However, $\hat \u-\u^*$ is not an infinitesimal rigid displacement. Indeed, assume by contradiction  that $\mathbb E(\u^*-\hat\u)\equiv 0$. Then $u_{1,x}^*\equiv u_{2,y}^*\equiv u_{1,y}^*+u_{2,x}^*\equiv 0$,
implying the existence of real constants $\bar a,\bar b,\bar c$ such that $u_1^*(x,y)=\bar a+\bar cy$ and $u_2^*(x,y)=\bar b-\bar cx$. 
Therefore, $(0,0,u_3^*)$ differs from $\u^*$ by an infinitesimal rigid displacements, and since $\mathcal G$ is invariant under the addition of infinitesimal rigid displacements, we obtain the minimality of $(0,0,u_3^*)$ for $\mathcal G$. But then the form \eqref{S0Lcontrex} of $\mathcal S_{\mathcal L}^0$ and the fact that $u_3^*$ depends only on $z$ directly imply
\[
\min_{H^1(\om,\mathbb R^3)}\mathcal G=\mathcal G(0,0,u_3^*)=4\int_\om(u^*_{3,z})^2\,d\x
\]
so that $u_{3,z}^*$ needs to be identically zero and we deduce that the trivial displacement field minimizes $\mathcal G$, so that
\[
0=\mathcal G(\mathbf 0)=\min_{H^1(\om,\R^3)}\mathcal G\le \min_{H^1(\om,\R^3)}\tilde{\mathcal G},
\] 
where $\tilde{\mathcal G}$ is defined by \eqref{auxfunct}. This contradicts Lemma \ref{lemma21} and concludes the proof.
\end{proof}
In the next statement, for every
 $\theta\in [-\pi,\pi]$ and for  $\mathbf R_\theta $ as in \ref{S0Lcontrex}, we use the notation 
\begin{equation*}
\displaystyle {\mathcal G}_\theta(\u)=\left\{\begin{array}{ll}\displaystyle 4\int_\om |\mathbb E(\u)|^2\,d\x- \mathcal L_{\mathbf R_\theta}(\u)\quad &\hbox{if} \ \u\in H^1(\om,\mathbb R^3)\\
&\\
 \ \!\!+\infty\quad &\hbox{otherwise in} \ W^{1,p}(\om,\mathbb R^3),
\end{array}\right.
\end{equation*}
where $\mathcal L_{\mathbf R}$ is defined by \eqref{elleerre}.
With this notation we clearly have $\mathcal G_0\equiv\mathcal E$ and $\mathcal G_{-\pi/2}\equiv\tilde{\mathcal G}$, where $\tilde{\mathcal G}$ is defined by \eqref{auxfunct}.
 \begin{proposition} Under the same assumptions of {\rm Proposition \ref{pro52}}, let $\u_0\in \argmin_{W^{1,p}(\om,\R^3)} \mathcal G_0$ and let $\u_{-\pi/2}\in\argmin_{W^{1,p}(\om,\R^3)} \mathcal G_{-\pi/2}$. Then,
 \begin{equation}\label{z11}
\min_{W^{1,p}(\om,\R^3)}\mathcal G_{\theta}=\cos^2\theta\,\mathcal G_0(\u_0)+\sin^2\theta\,\mathcal G_{-\pi/2}(\u_{-\pi/2})
\end{equation}
and
\begin{equation}\label{z22}
\displaystyle\min_{W^{1,p}(\om,\R^3)}\mathcal G=\min_{\theta\in[-\pi,\pi]}\;\;\min_{\u\in W^{1,p}(\om,\R^3)}\mathcal G_{\theta}(\u)=\mathcal G_{-\pi/2}(\u_{-\pi/2})<\mathcal G_0(\u_0)=\min_{W^{1,p}(\om,\R^3)}\mathcal E.
\end{equation}

 \end{proposition}

%

 \KKK

\begin{proof}
It is possible to check with a computation that the vector field $2r^{-1}\eta_*(r)(y,-x)$, where $r=\sqrt{x^2+y^2}$ and $\eta_*$ is defined by \eqref{etastar}, solves the problem
\begin{equation*}
\left\{\begin{array}{ll}-8\dv\widetilde{\mathbb{E}}(\tilde\u)=(\varphi_y,-\varphi_x)\qquad&\mbox{in $B$}\\
\widetilde{\mathbb E}( \tilde\u)\mathbf n=0\qquad&\mbox{on $\partial B$},
\end{array}\right.
\end{equation*}
where $\widetilde{\mathbb E}(\cdot)$ denotes the upper-left $2\times2$ submatrix of $\mathbb E(\cdot)$.
Moreover, thanks to the very same argument of the proof of Proposition \ref{pro52}, it is possible to find a minimizer  of functional $\mathcal G_{-\pi/2}$ in which the first two components do not depend on $z$ and the third component does not depend on $x,y$. Therefore, it is possible to find such a minimizer by decoupling the corresponding Euler-Lagrange equation for $\u=(u_1,u_2,u_3)$, i.e., 
\begin{equation*}
\left\{\begin{array}{ll}-8\dv{\mathbb{E}}(\u)=\mathbf R_{\pi/2}(\varphi_x,\varphi_y,\psi)\qquad&\mbox{in $\Omega$}\\
{\mathbb E}( \u)\mathbf n=0\qquad&\mbox{on $\partial \Omega$},
\end{array}\right.
\end{equation*}
in the above problem on $B$ for $\tilde \u=(u_1,u_2)$, and in the ordinary differential equation $-8u_3''
=\psi$ in the interval $(0,1)$, complemented by the conditions $u'_3(0)=u_3'(1)=0$, that gets solved, recalling \eqref{effe22}, by the function
\[
\Psi(z)=-\frac18\int_0^z\int_0^s\psi(t)\,dt\,ds.
\]
Therefore, a minimizer  of $\mathcal G_{-\pi/2}$ over $W^{1,p}(\om,\mathbb R^3)$ is given by
$
(2r^{-1}\eta_*(r)\,y,-2r^{-1}\eta_*(r)\,x, \Psi(z)).
$
Similarly,
$
(r^{-1}\eta_*(r)\,x,r^{-1}\eta_*(r)\,y, \Psi(z))
$
is a minimizer of $\mathcal G_0\equiv\mathcal E$, recalling that $r^{-1}\eta_*(r)(x,y)$ solves \eqref{elasticproblem} as seen in the proof of Theorem \ref{mainth3}. Hence, given $\u_0\in \argmin_{W^{1,p}(\om,\R^3)} \mathcal G_0$ and  $\u_{-\pi/2}\in\argmin_{W^{1,p}(\om,\R^3)} \mathcal G_{\pi/2}$, we may assume w.l.o.g. that
$$
\u_{-\pi/2}=
(2r^{-1}\eta_*(r)\,y,-2r^{-1}\eta_*(r)\,x, \Psi(z)),\qquad
\u_0=
(r^{-1}\eta_*(r)\,x,r^{-1}\eta_*(r)\,y, \Psi(z)),
$$
 and more generally we let
\[
\u_\theta:=\cos\theta(r^{-1}\eta_*(r)\,x,r^{-1}\eta_*(r)\,y, 0)-\sin\theta(2r^{-1}\eta_*(r)\,y,-2r^{-1}\eta_*(r)\,x, 0)+(0,0,\Psi(z)),
\] 
so that $\u_\theta$ solves
\begin{equation*}
\left\{\begin{array}{ll}-8\dv{\mathbb{E}}(\u)=\mathbf R^T_{\theta}(\varphi_x,\varphi_y,\psi)\qquad&\mbox{in $\Omega$}\\
{\mathbb E}( \u)\mathbf n=0\qquad&\mbox{on $\partial \Omega$},
\end{array}\right.
\end{equation*}
for any $\theta\in[-\pi,\pi]$, so that $\u_\theta$ is indeed a minimizer of $\mathcal G_\theta$.
Taking advantage of radiality and of the form of $\u_\theta$, it is easy to check that
\[
\int_\om \mathbb E(\u_0):\mathbb E(\u_{-\pi/2})\,d\x=0
\]
and that \eqref{z11} holds. 
But as shown in the proof of Theorem \ref{mainth3} we have
\[
\mathcal G_0(\u_0)=\min_{W^{1,p}(\om,\R^3)}\mathcal G_0=\min_{W^{1,p}(\om,\R^3)}\mathcal E>\min_{W^{1,p}(\om,\R^3)}\tilde{\mathcal G}=
\min_{W^{1,p}(\om,\R^3)}{\mathcal G}_{-\pi/2}=\mathcal G_{-\pi/2}(\u_{-\pi/2}),
\]
so that 
\[
\min_{\theta\in[-\pi,\pi]}\mathcal G(\u_\theta)=\min_{\theta\in[-\pi,\pi]}\cos^2\theta\,\mathcal G_0(\u_0)+\sin^2\theta\,\mathcal G_{-\pi/2}(\u_{-\pi/2})=\mathcal G_{-\pi/2}(\u_{-\pi/2}).
\]
We conclude that the optimal rotation realizing the maximum in the definition of $\mathcal G(\u_{-\pi/2})$ is given by $\mathbf R_{-\pi/2}$ and that \eqref{z22} holds true.
\end{proof}

We eventually remark that in view of the latter propositions (and in the same assumptions) and in view of Theorem \ref{rotheorem}, by taking $\theta=\pm\pi/2$, there is no gap between the minimal value of functional $\mathcal G_{\mathbf R_\theta}$ and that of functional $\mathcal E_{\mathbf R_\theta}$.

\subsection*{Acknowledgements} 
The authors acknowledge support from the MIUR-PRIN  project  No 2017TEXA3H.
The authors are members of the
GNAMPA group of the Istituto Nazionale di Alta Matematica (INdAM).

\end{document}